\documentclass[12pt]{amsart}
\usepackage{faktor}
\usepackage{hyperref}
\hypersetup{
	hypertexnames=false,
	colorlinks=true,
	linktoc=all,
	linkcolor=black,
}

\usepackage{changepage}

\usepackage{pictexwd,dcpic}
\usepackage[justification=centering]{caption}
\usepackage{graphicx}
\usepackage{latexsym,amssymb,amscd}
\usepackage{amsmath}
\usepackage[all]{xy}
\usepackage{amsthm}
\usepackage{mathrsfs}
\usepackage{times,enumerate}
\usepackage[usenames,dvipsnames]{color}
\usepackage{blkarray}
\usepackage{amsmath}
\usepackage{faktor}
\usepackage{tikz}

\usetikzlibrary{arrows,decorations.pathmorphing,backgrounds,positioning,fit,matrix}
\usepackage{tikz-cd}
\usepackage{float}
\usepackage{hhline}
\usepackage[english]{babel}

\usepackage{amssymb}

\theoremstyle{plain}
\newtheorem{theorem}{Theorem}[section]
\newtheorem{lemma}[theorem]{Lemma}
\newtheorem{corollary}[theorem]{Corollary}
\newtheorem{proposition}[theorem]{Propositition}
\newtheorem{example}[theorem]{Example}

\theoremstyle{definition}
\newtheorem{definition}[theorem]{Definition}

\theoremstyle{remark}
\newtheorem{remark}[theorem]{Remark}

\makeatletter
\def\@setcopyright{}
\def\serieslogo@{}
\makeatother

\let\OLDthebibliography\thebibliography
\renewcommand\thebibliography[1]{
	\OLDthebibliography{#1}
	\setlength{\parskip}{5pt}
	\setlength{\itemsep}{0pt plus 0.ex}
}

\begin{document}
	   \author{Spyridon Afentoulidis-Almpanis}
	
	\address{Department of Mathematics, University of Lorraine, 3 rue Augustin
		Fresnel, 57070 Metz}
	
	\email{spyridon.afentoulidis-almpanis@univ-lorraine.fr}
	
	
	\title[Noncubic Dirac operators for finite dimensional modules]{Noncubic Dirac operators for finite dimensional modules}
	
	
	


\maketitle

\begin{abstract}
We study the decomposition into irreducibles of the kernel of noncubic Dirac operators attached to finite-dimensional modules. We compare this decomposition with features of Kostant's cubic Dirac operator. In particular, we show that the kernel of noncubic Dirac operators need not contain full isotypic components. The cases of classical and exceptional complex Lie algebras are studied in detail. As a by-product, we deduce some information on the kernel of noncubic geometric Dirac operators acting on sections over compact manifolds studied by Slebarski.
\end{abstract}

\tableofcontents

 \section{Introduction}

Dirac operators play an important role in representation theory of Lie groups. For example, every discrete series representation of a noncompact connected semisimple real Lie group can be realized as the $L^2$-kernel of a Dirac operator acting on sections of a bundle over a Riemannian symmetric space $G/K$ \cite{parthaThesis,atiyahsm}. On the other hand, smooth vectors of principal series representations can be embedded in smooth kernels of Dirac operators on a reductive homogeneous space $G/H$ \cite{mehdizierauAdv}.

In the late $1990$'s, in a series of lectures at MIT, Vogan defined an algebraic Dirac operator \cite{vogantalks}. 
More precisely, if $\mathfrak{g}$ is a complex semisimple Lie algebra and $\mathfrak{k}$ is a maximal compact subalgebra of $\mathfrak{g}$, to each $(\mathfrak{g},K)$-module $(\pi,V)$, he associated an operator $D_{\mathfrak{g},\mathfrak{k}}(V):V\otimes S\rightarrow V\otimes S$ given by
\begin{equation*}\label{kos}
D_{{\mathfrak g},{\mathfrak k}}(V)=\sqrt{2}\sum_j \pi(X_j)\otimes\gamma(X_j),
\end{equation*}
where $\{X_j\}$ is an orthonormal basis of ${\mathfrak p}:=\mathfrak{k}^\perp$, $S$ is a spin module for the Clifford algebra $\mathbf{C}(\mathfrak{p})$ of $\mathfrak{p}$ and $\gamma$ stands for the action of $\mathbf{C}(\mathfrak{p})$ on $S$. Here the orthogonal is taken with respect to the Killing form $\langle\cdot,\cdot\rangle$ of $\mathfrak{g}$. The factor $\sqrt{2}$ occuring in front of $D_{\mathfrak{g},\mathfrak{k}}(V)$ will be explained further later in the text.
The operator $D_{\mathfrak{g},\mathfrak{k}}(V)$ does not depend on the choice of the basis $\{X_i\}$ and, provided that $V$ is unitary, is self-adjoint with respect to a Hermitian form defined on $V\otimes S$. Moreover, its square differs from the Casimir element by a scalar:
\begin{equation}\label{diracsquare}
D_{{\mathfrak g},{\mathfrak k}}(V)^2=\pi(\Omega_{\mathfrak g})\otimes 1-(\pi\otimes\gamma)(\Omega_{{\mathfrak k}_\Delta})+(\lVert\rho\rVert^2-\lVert\rho_{\mathfrak k}\rVert^2)1\otimes 1,
\end{equation}
where $\Omega_{\mathfrak g}$ (respectively $\Omega_{\mathfrak k}$) denotes the Casimir element in the enveloping algebra $U({\mathfrak g})$ (respectively $U({\mathfrak k})$) of $\mathfrak{g}$ (respectively $\mathfrak{k}$), $\rho$ (respectively $\rho_{\mathfrak k}$) is half the sum of the positive roots in a positive system $\Delta^+$ (respectively $\Delta_{\mathfrak k}^+\subset\Delta^+$) for roots in $\mathfrak{g}$ (respectively ${\mathfrak k}$), $\lVert\cdot\rVert$ is the norm induced by the Killing form of ${\mathfrak g}$, and $(\cdot)_\Delta$ is the diagonal embedding 
\begin{equation*}
(\cdot)_\Delta:{\mathfrak k}\longrightarrow U({\mathfrak g})\otimes\mathbf{C}({\mathfrak p})
\end{equation*}
of ${\mathfrak k}$ into $U({\mathfrak g})\otimes\mathbf{C}({\mathfrak p})$, given by the embedding ${\mathfrak k}\subset{\mathfrak g}$ in $U({\mathfrak g})$ and the embedding of $\mathfrak{k}$ in $\mathbf{C}({\mathfrak p})$. Equipped with this action, $V\otimes S$ is a Lie algebra representation of $\mathfrak{k}$ and the operator $D_{\mathfrak{g},\mathfrak{k}}(V)$ turns out to be $\widetilde{K}$-equivariant, where $\widetilde{K}$ denotes the spin double cover of $K$. Consequently, the kernel $\ker D_{\mathfrak{g},\mathfrak{k}}(V)$ and the image $\mathrm{im}\hspace{.7mm} D_{\mathfrak{g},\mathfrak{k}}(V)$ of $D_{\mathfrak{g},\mathfrak{k}}(V)$ are $\widetilde{K}$-modules.  Vogan defined the Dirac cohomology of a $(\mathfrak{g},\widetilde{K})$-module $V$ to be the $\widetilde{K}$-module 
\begin{equation}\label{diraccohom}
H_D^{{\mathfrak g},{\mathfrak k}}(V)=\ker D_{{\mathfrak g},{\mathfrak k}}(V)/\big(\text{im}\hspace{.7mm}D_{{\mathfrak g},{\mathfrak k}}(V)\cap\ker D_{{\mathfrak g},{\mathfrak k}}(V)\big).
\end{equation}
Assuming that $V$ is irreducible, he conjectured that if $H_D^{{\mathfrak g},{\mathfrak k}}(V)$ contains a $\widetilde{K}$-module with highest weight $\beta$ then $V$ has infinitesimal character $\beta+\rho_{\mathfrak k}$. In other words, the infinitesimal character of a module can be recovered from its Dirac cohomology. Vogan's conjecture was first proved by Huang and Pand\v zi\'c for equal rank pairs  $({\mathfrak g},{\mathfrak k})$ in \cite{huangpandzic}.

In 1999, Goette \cite{goette} and independently Kostant \cite{Kostant-1999} generalized the above construction in the case of equal rank pairs $(\mathfrak{g},\mathfrak{h})$ where $\mathfrak{g}$ is as above and $\mathfrak{h}$ is a reductive Lie subalgebra of $\mathfrak{g}$ such that the restriction of the Killing form of $\mathfrak{g}$ to $\mathfrak{h}\times\mathfrak{h}$ remains nondegenerate. More precisely, if $\mathfrak{q}:=\mathfrak{h}^\perp$ with respect to $\langle\cdot,\cdot\rangle$ and $(\pi,V)$ is, as before, a $(\mathfrak{g},K)$-module, Goette and Kostant defined an operator 
\begin{equation}\label{diracintro}
D_{{\mathfrak g},{\mathfrak h}}(V):V\otimes S\rightarrow V\otimes S
\end{equation}
given by
\begin{equation*}D_{{\mathfrak g},{\mathfrak h}}(V)=\sqrt{2}\{\sum_j \pi(X_j)\otimes\gamma(X_j)-1\otimes\gamma(c)\},
\end{equation*}
where $\{X_j\}$ is an orthonormal basis of ${\mathfrak q}$ and $\gamma(c)$ is given by
\begin{equation}\label{cubictermm}
\gamma(c)=\frac{1}{6}\sum_{i,j,k} \langle X_i,[X_j,X_k]\rangle\gamma(X_i)\gamma(X_j)\gamma(X_k).
\end{equation}
We note that in the case where $\mathfrak{h}$ is a symmetric Lie algebra, $\gamma(c)=0$.
The operator $D_{\mathfrak{g},\mathfrak{h}}(V)$, known as Kostant's cubic Dirac operator, has the same properties as $D_{\mathfrak{g},\mathfrak{k}}(V)$. Namely, $D_{\mathfrak{g},\mathfrak{h}}(V)$  
does not depend on the choice of the basis $\{X_i\}$, is $\mathfrak{h}$-equivariant with respect to the action of $\mathfrak{h}$ on $V\otimes S$ and, provided that $V$ is unitary, is self-adjoint with respect to a Hermitian 
form defined on $V\otimes S$. Goette and Kostant considered the extra cubic term $1\otimes \gamma(c)$ in the definition of the above operator in order to have a convenient formula for the square $D_{\mathfrak{g},\mathfrak{h}}(V)^2$ of $D_{\mathfrak{g},\mathfrak{h}}(V)$:
\begin{equation}\label{Dsquare}
D_{\mathfrak{g},\mathfrak{h}}(V)^2=\pi(\Omega_\mathfrak{g})\otimes1-(\pi\otimes\gamma)(\Omega_{\mathfrak{h}_\Delta})+(\lVert\rho\rVert^2-\lVert\rho_\mathfrak{h}\rVert^2)1\otimes 1,
\end{equation}
where $\rho_\mathfrak{h}$ is the half-sum of the positive roots of some positive system $\Delta_\mathfrak{h}^+\subset\Delta_\mathfrak{h}$ for $\Delta_\mathfrak{h}:=\Delta(\mathfrak{h},\mathfrak{t})$.

In the case where $V$ is finite-dimensional of highest weight $\lambda$, $D_{\mathfrak{g},\mathfrak{h}}(V)$ is self-adjoint so that $\text{im}\hspace{0.7mm}D_{{\mathfrak g},{\mathfrak h}}(V)\cap\ker D_{{\mathfrak g},{\mathfrak h}}(V)$ is trivial \cite[Proposition 3.2.9 and Corollary 3.2.10]{thesis}. Hence the Dirac cohomology $H_D^{{\mathfrak g},{\mathfrak h}}(V)$ of $V$ coincides with the kernel $\ker D_{{\mathfrak g},{\mathfrak h}}(V)$ for which
Kostant gave a complete decomposition into irreducibles \cite{Kostant-1999}:
\begin{equation}\label{kernel}
\ker D_{\mathfrak{g},\mathfrak{h}}(V)=\bigoplus_{w\in W^1} F_{w(\lambda+\rho)-\rho_{\mathfrak{h}}},
\end{equation}
where
\begin{equation*}
W^1:=\{w\in W\mid \Delta_\mathfrak{h}^+\subseteq w\Delta^+\},
\end{equation*}
$W$ is the Weyl group of $\Delta$,
and $F_{w(\lambda+\rho)-\rho_{\mathfrak{h}}}$ is the finite-dimensional irreducible representation of $\mathfrak{h}$ with highest weight $w(\lambda+\rho)-\rho_{\mathfrak{h}}$.
A similar formula was proved in the unequal rank situation in \cite{kang-pandzic} for the pair $({\mathfrak g},{\mathfrak k})$ and in \cite{Mehdi-2014} for pairs $({\mathfrak g},{\mathfrak h})$. In particular, in the unequal rank case, the decomposition of $H_D^{{\mathfrak g},{\mathfrak k}}(V)$ (respectively $H_D^{{\mathfrak g},{\mathfrak h}}(V)$) into irreducibles is no longer multiplicity free.

In this paper, we consider equal rank pairs $(\mathfrak{g},\mathfrak{h})$ as above with $\mathfrak{h}$ nonsymmetric. Then, $\gamma(c)\neq0$ and we define a family $\{D^t_{\mathfrak{g},\mathfrak{h}}(V)\}_{t\in[0,2]}$ of Dirac-type operators 
\begin{equation*}
D^t_{\mathfrak{g},\mathfrak{h}}(V):V\otimes S\rightarrow V\otimes S
\end{equation*}
given by
\begin{equation*}
D^t_{\mathfrak{g},\mathfrak{h}}(V):=D_{\mathfrak{g},\mathfrak{h}}(V)+(1-t)\sqrt{2}\otimes \gamma(c).
\end{equation*}
These operators, which we shall call $t$-noncubic Dirac operators in the case when $t\neq1$, all are $\mathfrak{h}$-equivariant so that their kernels are representations of the Lie algebra $\mathfrak{h}$.
For $t=1$, one obtains the Kostant's cubic Dirac operator $D_{\mathfrak{g},\mathfrak{h}}(V)$ while for $t=0$, one obtains the operator $\widehat{D}_{\mathfrak{g},\mathfrak{h}}(V):=D^0_{\mathfrak{g},\mathfrak{h}}(V)$ that we call the noncubic Dirac operator. These operators are a representation-theoretic analogue of a family of invariant differential Dirac operators acting on sections over manifolds studied by Slebarski \cite{slebarski,slebarski2,slebarski3} and Agricola \cite{Ilka} from a differential geometry viewpoint. We are interested in studying the kernel of $D^t_{\mathfrak{g},\mathfrak{h}}(V)$, $t\in[0,2]$, in the case when $V$ is finite-dimensional. The basic difficulty one has to address is the fact that, as we will see, for $t\neq 1$, the square of $D^t_{\mathfrak{g},\mathfrak{h}}(V)$ has not the "good" form \eqref{diracsquare} which is crucial for proving \eqref{kernel}. As a result, $\ker D^t_{\mathfrak{g},\mathfrak{h}}(V)$ may be very different from $\ker D_{\mathfrak{g},\mathfrak{h}}(V)$. Let us illustrate the situation with the following example.

\begin{example}\label{firstexample} Let $\mathfrak{g}=\mathfrak{sl}(4,\mathbb{C})$, i.e. the Lie algebra of $4-$by$-4$ traceless complex matrices, and $\mathfrak{t}$ a Cartan subalgebra of $\mathfrak{g}$. The root system of $\mathfrak{g}$ is
	\begin{equation*}
	\Delta=\{\varepsilon_i-\varepsilon_j\mid 1\leq i\neq j\leq 4\}
	\end{equation*}
	and 
	\begin{equation*}
	\Delta^+=\{\varepsilon_i-\varepsilon_j\mid 1\leq i<j\leq 4\}
	\end{equation*}
	is a positive system for $\Delta$. For every $\alpha\in\Delta^+$, choose nonzero root vectors $E_{\pm\alpha}\in\mathfrak{g}_{\pm\alpha}$ and set
	\begin{equation*}
	H_\alpha:=\dfrac{[E_\alpha,E_{-\alpha}]}{\langle E_\alpha,E_{-\alpha}\rangle}.
	\end{equation*}
	The following table relates the kernels of the cubic and the noncubic Dirac operators in the case where  $\mathfrak{h}$ is the indicated Lie subalgebra of $\mathfrak{g}$ and $V$ is the standard representation of $\mathfrak{g}$.
	
	{ \renewcommand*{\arraystretch}{2}
		\begin{table}[H]
			\centering
			\begin{adjustwidth}{-6mm}{-6mm}
				\begin{tabular}{|c|c|c|}
					\hline
					Ranks &$\mathfrak{h}$ &$\ker D_{\mathfrak{g},\mathfrak{h}}(V)$ $ \textbf{vs.}$ $\ker\widehat{D}_{\mathfrak{g},\mathfrak{h}}(V)$ \\
					\hhline{|=|=|=|}
					$\mathrm{rk}_\mathbb{C}(\mathfrak{g})=\mathrm{rk}_\mathbb{C}(\mathfrak{h})$&$\mathfrak{t}$&
					$\ker D_{\mathfrak{g},\mathfrak{h}}(V)\varsubsetneq\ker\widehat{D}_{\mathfrak{g},\mathfrak{h}}(V)$\\
					$\mathrm{rk}_\mathbb{C}(\mathfrak{g})>\mathrm{rk}_\mathbb{C}(\mathfrak{h})$&$\mathbb{C}\{H_{\varepsilon_1-\varepsilon_3}\}\oplus\mathfrak{g}_{\varepsilon_1-\varepsilon_3}\oplus\mathfrak{g}_{-\varepsilon_1+\varepsilon_3}$&$\ker D_{\mathfrak{g},\mathfrak{h}}(V)\subseteq\ker\widehat{D}_{\mathfrak{g},\mathfrak{h}}(V)$\\
					$\mathrm{rk}_\mathbb{C}(\mathfrak{g})>\mathrm{rk}_\mathbb{C}(\mathfrak{h})$&$\mathbb{C}\{5H_{\varepsilon_1-\varepsilon_2}+4H_{\varepsilon_2-\varepsilon_3}\}$&$\ker D_{\mathfrak{g},\mathfrak{h}}(V)\nsubseteq\ker\widehat{D}_{\mathfrak{g},\mathfrak{h}}(V)$\\
					$\mathrm{rk}_\mathbb{C}(\mathfrak{g})=\mathrm{rk}_\mathbb{C}(\mathfrak{h})$&$\mathfrak{t}\oplus\bigoplus\limits_{1\leq i,j\leq 3}\mathfrak{g}_{\varepsilon_i-\varepsilon_j}$&$\ker D_{\mathfrak{g},\mathfrak{h}}(V)=\ker \widehat{D}_{\mathfrak{g},\mathfrak{h}}(V)$\\
					\hline
				\end{tabular}
			\end{adjustwidth}
			\vspace{2mm}
			\captionsetup{justification=centering}
			\caption{Relation of kernels of cubic and noncubic Dirac operators for the standard representation}
		\end{table}
	}
\end{example}

This paper aims to study the decomposition into irreducible $\mathfrak{h}$-represe\-ntations of the kernel of $t$-noncubic Dirac operators $D^t_{\mathfrak{g},\mathfrak{h}}(V)$ when $V$ is finite-dimensional. In the case of equal rank $\mathfrak{g}$ and $\mathfrak{h}$ and for $t$ different from $0$ and $2$, we show that the kernel $\ker D^t_{\mathfrak{g},\mathfrak{h}}(V)$ of the $t$-noncubic Dirac operator $D^t_{\mathfrak{g},\mathfrak{h}}(V)$ coincides with the kernel $\ker D_{\mathfrak{g},\mathfrak{h}}(V)$ of the cubic one (Theorem \ref{kerneltcubic}). For the extreme values $0$ and $2$, we give a decomposition in the case when $\mathfrak{h}$ is a Cartan subalgebra $\mathfrak{t}$ of $\mathfrak{g}$ (Theorem \ref{thmbasic}) while for general equal rank $\mathfrak{h}$ we show that the kernel of the cubic Dirac operator is actually contained in the kernel of the noncubic one (Proposition \ref{kerinclusion}) and differs from the other cases in the sense that full isotypic components need not lie in $\ker \widehat{D}_{\mathfrak{g},\mathfrak{h}}(V)$ (Example \ref{nonpolyn}). 
An essential ingredient of our study are the Weyl's inequalities which we recall in Section \ref{weylssubsection}. Our decomposition for $\ker \widehat{D}_{\mathfrak{g},\mathfrak{t}}(V)$ is made more explicit for complex classical and exceptional Lie algebras (Section \ref{sectionclassical} and \ref{sectionexceptional} respectively). Finally as a by-product, we apply our results to geometric Dirac operators on compact manifolds via a duality principle between algebraic and geometric Dirac operators (Theorems \ref{thm62} and \ref{thm63}).

The paper is organized as follows. In Section \ref{Section spin rep}, we recall the necessary theory of Clifford algebras and spin modules while in Sections \ref{main}-\ref{mainfin} we state and prove our main results concerning the kernel of $t$-noncubic Dirac operators. In Sections \ref{sectionclassical} and \ref{sectionexceptional}, we adapt these results in the case of classical and exceptional, respectively, Lie algebras. Finally, in Section \ref{sectionSleb}, we discuss applications to Slebarski's Dirac operators.

This work is part of my Ph.D. thesis at the University of Lorraine. I express my deep gratitude to my thesis advisor Professor Salah Mehdi for his guidance and encouragement.


\section{Spin representation}\label{Section spin rep} Let $\mathfrak{g}$ be a complex semisimple Lie algebra, $\langle\cdot,\cdot\rangle$ its Killing form and $\mathfrak{h}$ a Lie subalgebra of $\mathfrak{g}$ such that:
\begin{equation}\label{conditionh}
\begin{cases}
\mathfrak{h}\text{ is reductive}\\
\mathrm{rk}_\mathbb{C}(\mathfrak{g})=\mathrm{rk}_\mathbb{C}(\mathfrak{h}).
\end{cases}
\end{equation}
Then, the restriction $\langle\cdot,\cdot\rangle_{\mid\mathfrak{h}\times\mathfrak{h}}$ of $\langle\cdot,\cdot\rangle$ on $\mathfrak{h}$  is nondegenerate and there is therefore a decomposition:
\begin{equation*}
\mathfrak{g}=\mathfrak{h}\oplus\mathfrak{q},
\end{equation*}   
where $\mathfrak{q}$ is the orthogonal complement of $\mathfrak{h}$ in $\mathfrak{g}$ with respect to $\langle\cdot,\cdot\rangle$. Consider the Clifford algebra $\mathbf{C}(\mathfrak{q})$ of $\mathfrak{q}$, i.e. the quotient of the tensor algebra $T(\mathfrak{q})$ of $\mathfrak{q}$ by the ideal generated by elements 
\begin{equation*}
X\otimes Y+Y\otimes X-\langle X,Y\rangle \text{ for } X,Y\in\mathfrak{q}.
\end{equation*} 
Let $\gamma:\mathfrak{q}\rightarrow \mathbf{C}(\mathfrak{q})$ be the composition of the injection of $\mathfrak{q}$ into $T(\mathfrak{q})$ and the projection onto $\mathbf{C}(\mathfrak{q})$.
The algebra $\mathbf{C}(\mathfrak{q})$, up to equivalence, has
a unique irreducible module S, called space of spinors for $\mathbf{C}(\mathfrak{q})$ \cite{goodman}. Moreover, there is a Hermitian form $\langle\cdot,\cdot\rangle_S$ on $S$ such that the operator $\gamma(X)$ is self-adjoint for every $X\in\mathfrak{q}$ \cite{thesis}. 

Let $\mathfrak{t}$ be a Cartan subalgebra of $\mathfrak{g}$ contained in $\mathfrak{h}$. The root system $\Delta:=\Delta(\mathfrak{g},\mathfrak{t})$ splits into $\Delta=\Delta_\mathfrak{h}\sqcup\Delta_\mathfrak{q}$ so that 
\begin{equation*}
\mathfrak{h}=\mathfrak{t}\oplus\bigoplus_{\alpha\in\Delta_\mathfrak{h}}\mathfrak{g}_\alpha \text{ and }
\mathfrak{q}=\bigoplus_{\beta\in\Delta_\mathfrak{q}}\mathfrak{g}_\beta.
\end{equation*}
Choose a positive system $\Delta^+\subset\Delta$ for $\Delta$ and set $\Delta_\mathfrak{h}^+:=\Delta^+\cap\Delta_\mathfrak{h}$ to be a positive system for the root system $\Delta_\mathfrak{h}$, and $\Delta_\mathfrak{q}^+:=\Delta^+\cap\Delta_\mathfrak{q}$.
In particular, $\mathfrak{q}=\mathfrak{q}^+\oplus\mathfrak{q}^-$ where
\begin{equation*}
\mathfrak{q}^+:=\bigoplus_{\beta\in\Delta_\mathfrak{q}^+}\mathfrak{g}_{\beta}\quad\text{and}\quad\mathfrak{q}^-:=\bigoplus_{\beta\in\Delta_\mathfrak{q}^+}\mathfrak{g}_{-\beta}
\end{equation*}
are dual maximal isotropic subspaces of $\mathfrak{q}$.
Now $S$ can be chosen to be 
\begin{equation*}
S:=\bigwedge \mathfrak{q}^-
\end{equation*}
with the $\mathbf{C}(\mathfrak{q})$-action being given by the Clifford multiplication.
Moreover, there is a Lie algebra monomorphism $\varphi:\mathfrak{so}(\mathfrak{q})\rightarrow \mathbf{C}(\mathfrak{q})$ so that $S$ becomes a Lie algebra representation of $\mathfrak{h}$, known as the spin representation, via the composition map
\begin{equation}\label{haction}
\mathfrak{h}\overset{\mathrm{ad}}{\longrightarrow} \mathfrak{so}(\mathfrak{q})\overset{\varphi}{\longrightarrow} \mathbf{C}(\mathfrak{q})\overset{\gamma}{\longrightarrow} \mathrm{End}(S).
\end{equation}
As an $\mathfrak{h}$-representation, $S$ can be decomposed into a direct sum of its weight subspaces. More precisely, let $\beta_1,\ldots,\beta_l$
be the positive weights of $\mathfrak{q}$, repeated according to their multiplicities, and $e_{\pm 1},\ldots,e_{\pm k}$ the corresponding weight vectors of weights $\pm\beta_1,\ldots,\pm\beta_l$ respectively, such that 

\begin{equation*}
\langle e_i,e_j\rangle=
\begin{cases}
1&\text{if }i+j=0\\
0& \text{otherwise.}
\end{cases}
\end{equation*}
If $I=(i_1,\ldots,i_s)$ with $1\leq s\leq k $, the element
\begin{equation}\label{uI}
u_I:=e_{-i_1}\wedge\ldots\wedge e_{-i_s}\in S
\end{equation}
is a weight vector of $S$ of weight
\begin{equation}\label{Sweight}\rho-\rho_\mathfrak{h}-\sum\limits_{i\in I}\beta_i=\frac{1}{2}\big\{\sum\limits_{i\notin I}\beta_i-\sum\limits_{i\in I}\beta_i\big\},
\end{equation}
where $\rho$ (respectively $\rho_\mathfrak{h}$) is the half-sum of the positive roots of $\Delta^+$ (respectively $\Delta_\mathfrak{h}^+$).
Here, by abuse of notation, we write $i\in I$ if $i\in\{i_1,\ldots,i_s\}$.

\section{Noncubic Dirac operators}

\subsection{$t$-noncubic Dirac operators}\label{main}
%
%
We keep the previous notation. Namely, $\mathfrak{g}$ is a complex semisimple Lie algebra, $\mathfrak{h}$ a Lie subalgebra of $\mathfrak{g}$ satisfying condition \eqref{conditionh} with $\mathfrak{t}$ being a common Cartan subalgebra of $\mathfrak{g}$ and $\mathfrak{h}$, while $\mathfrak{q}$ is the orthogonal complement of $\mathfrak{h}$ with respect to the Killing form $\langle\cdot,\cdot\rangle$ of $\mathfrak{g}$.
To every finite-dimensional  representation $(\pi,V)$ of $\mathfrak{g}$, we attach a family of operators
\begin{equation*}
D^t_{\mathfrak{g},\mathfrak{h}}(V):V\otimes S\rightarrow V\otimes S,\quad t\in[0,2]
\end{equation*}
defined by
\begin{equation*}
D^t_{\mathfrak{g},\mathfrak{h}}(V)=\sqrt{2}\big\{\sum_{i,j}\langle\tilde{e}_i,\tilde{e}_j\rangle\pi(e_i)\otimes \gamma(e_j)-t\big(1\otimes\gamma(c)\big)\big\},
\end{equation*}
where $\{\tilde{e}_i\}$ and $\{e_i\}$ are dual bases of $\mathfrak{q}$ with respect to the Killing form $\langle\cdot,\cdot\rangle$ and $\gamma(c)$ is the cubic term
\begin{equation}\label{cub}
\gamma(c)=\frac{1}{6}\sum_{i,j,k}\langle\tilde{e}_i,[\tilde{e}_j,\tilde{e}_k]\rangle\gamma(e_i)\gamma(e_j)\gamma(e_k).
\end{equation}

\begin{definition}[Noncubic Dirac operators]\label{noncubic}\index{Dirac operator!noncubic}
	For every $t\in[0,1)\cup(1,2]$, the operator $D^t_{\mathfrak{g},\mathfrak{h}}(V)$
	is called $t$-noncubic Dirac operator. In case when $t=0$, the operator $\widehat{D}_{\mathfrak{g},\mathfrak{h}}(V):=D^0_{\mathfrak{g},\mathfrak{h}}(V)$ is simply called noncubic Dirac operator.
\end{definition}

\begin{remark}\label{remarkcoinc} Note that for $t=1$, one obtains the cubic Dirac operator $D_{\mathfrak{g},\mathfrak{h}}(V)$ of \eqref{diracintro}. Moreover, in the case where $\mathfrak{h}$ is symmetric, i.e. $\mathfrak{h}$ is the set of fixed elements by an involution $\sigma$ of $\mathfrak{g}$, the cubic term $\gamma(c)$ vanishes so that every operator $D^t_{\mathfrak{g},\mathfrak{h}}(V)$ coincides with $D_{\mathfrak{g},\mathfrak{h}}(V)$.
\end{remark}


As for $D_{\mathfrak{g},\mathfrak{h}}(V)$, one can check that $D^t_{\mathfrak{g},\mathfrak{h}}(V)$ is independent of the choice of the bases, $\mathfrak{h}$-equivariant and self-adjoint with respect to a Hermitian form defined on $V\otimes S$ \cite{thesis}. In particular, the kernel $\ker D^t_{\mathfrak{g},\mathfrak{h}}(V)$ of $D^t_{\mathfrak{g},\mathfrak{h}}(V)$ is a representation of $\mathfrak{h}$. We are interested in describing this kernel. Nevertheless, the main difficulty is that, unlike $D_{\mathfrak{g},\mathfrak{h}}(V)$, provided that $\mathfrak{h}$ is not symmetric, $D^t_{\mathfrak{g},\mathfrak{h}}(V)^2$ has not the good form \eqref{Dsquare} of $D_{\mathfrak{g},\mathfrak{h}}(V)^2$, which is crucial for the description \eqref{kernel} of the kernel $\ker D_{\mathfrak{g},\mathfrak{h}}(V)$ of $D_{\mathfrak{g},\mathfrak{h}}(V)$.
%
%
A first result in this direction is the following.

\begin{proposition}\label{kerinclusion}
	For every $t\in[0,2]$, we have the inclusion
	\begin{equation*}
	\ker D_{\mathfrak{g},\mathfrak{h}}(V)\subseteq \ker D^t_{\mathfrak{g},\mathfrak{h}}(V).
	\end{equation*}
\end{proposition}

\begin{proof} It suffices to show that $\ker D_{\mathfrak{g},\mathfrak{h}}(V)\subseteq\ker\widehat{D}_{\mathfrak{g},\mathfrak{h}}(V)$. Indeed, provided this is true, if $x\in\ker D_{\mathfrak{g},\mathfrak{h}}(V)$, then 
	\begin{align*}
	D^t_{\mathfrak{g},\mathfrak{h}}(V)(x)&=\big(D_{\mathfrak{g},\mathfrak{h}}(V)+(1-t)\sqrt{2}\otimes \gamma(c)\big)(x)\\
	&=tD_{\mathfrak{g},\mathfrak{h}}(V)(x)+(1-t)\widehat{D}_{\mathfrak{g},\mathfrak{h}}(V)(x)\\
	&=0
	\end{align*}
	and so $\ker D_{\mathfrak{g},\mathfrak{h}}(V)\subseteq\ker D^t_{\mathfrak{g},\mathfrak{h}}(V)$.

	Let $V$ be of highest weight $\lambda\in\mathfrak{t}^*$ and $v_\lambda$ be the highest weight vector of $V$. Moreover, for every algebraically integral $\Delta_\mathfrak{h}^+$-dominant element $\mu\in\mathfrak{t}^*$, let $F_{\mu}$ be the corresponding finite-dimensional irreducible $\mathfrak{h}$-representation of highest weight $\mu\in\mathfrak{t}^*$. 
	According to Kostant's formula \eqref{kernel}, we will be finished once we show that, for every $w$ in
	$
	W^1=\{w\in W\mid \Delta_\mathfrak{h}^+\subseteq w\Delta^+\}$
	the representation $F_{w(\lambda+\rho)-\rho_\mathfrak{h}}$ belongs to $\ker\widehat{D}_{\mathfrak{g},\mathfrak{h}}(V)$. 
	
	First, we consider $F_{\lambda+\rho-\rho_{\mathfrak{h}}}$. 
	By definition, $1\in S$ is a maximal weight vector of $S$. 
	Let $\{e_\alpha \}$ and $\{e_{-\alpha} \}$ be root vector bases of $\mathfrak{q}^+$ and $\mathfrak{q}^-$ respectively such that $\langle e_\alpha,e_{-\alpha}\rangle=1$. Then the noncubic Dirac operator is given by
	\begin{equation*}
	\widehat{D}_{\mathfrak{g},\mathfrak{h}}(V)=\sqrt{2}\sum_{\alpha\in\Delta^+_\mathfrak{q}}e_\alpha\otimes\gamma(e_{-\alpha})
	+\sqrt{2}\sum_{\alpha\in\Delta^+_\mathfrak{q}}
	e_{-\alpha}\otimes\gamma(e_{\alpha}).
	\end{equation*}
	Set \begin{equation*}\widehat{D}_1:=\sqrt{2}\sum_{\alpha\in\Delta^+_\mathfrak{q}}e_\alpha\otimes\gamma(e_{-\alpha})
	\end{equation*} and \begin{equation*}\widehat{D}_2:=\sqrt{2}\sum_{\alpha\in\Delta^+_\mathfrak{q}}e_{-\alpha}\otimes\gamma(e_{\alpha}).\end{equation*}
	Then $\widehat{D}_1$ (respectively $\widehat{D}_2$) acts trivially on the first factor (respectively second) of $v_\lambda\otimes1$ and thus $\widehat{D}_{\mathfrak{g},\mathfrak{h}}(V)$ acts trivially on $v_\lambda\otimes 1$. Using the facts that $F_{\lambda+\rho-\rho_{\mathfrak{h}}}$ is irreducible and $\widehat{D}_{\mathfrak{g},\mathfrak{h}}(V)$ is $\mathfrak{h}$-equivariant, one deduces that the whole representation $F_{\lambda+\rho-\rho_{\mathfrak{h}}}$, generated by $v_\lambda\otimes 1$ as a $U(\mathfrak{h})$-module, belongs to $\ker\widehat{D}_{\mathfrak{g},\mathfrak{h}}(V)$. 
	
	To treat the general case of a subrepresentation $F_{w(\lambda+\rho)-\rho_{\mathfrak{h}}}$ of $\ker D_{\mathfrak{g},\mathfrak{h}}(V)$, for $w\in W^1$, we choose $w\Delta^+$ to be the positive system of $\Delta$. Note that by the definition of $W^1$, $\Delta^+_\mathfrak{h}\subseteq w\Delta^+$. Therefore $w\lambda$ is the highest weight of $V$ with respect to this choice while $w(\rho-\rho_{\mathfrak{h}})=w\rho-\rho_{\mathfrak{h}}$ is a maximal weight of $S$ with respect to $\Delta_\mathfrak{h}^+$. 
	The last statement can be seen by using $w\Delta^+$ instead of $\Delta^+$ in the choice of the space of spinors for $\mathbf{C}(\mathfrak{q})$. More precisely, we define $S$ using again the negative roots of $\mathfrak{q}$ but now with respect to $w\Delta^+$. Note that, due to \cite[Theorem 6.1.3]{goodman}, the space of spinors $S$ does not depend on the choice of the positive system for $\Delta$.
	According to formula \eqref{Sweight}, $w\rho-\rho_\mathfrak{h}$ is a highest weight of $S$.
	Therefore, if $v_{w\lambda}$ and $u_{w\rho-\rho_{\mathfrak{h}}}$ are the corresponding highest weight vectors of $V$ and $S$ respectively, one has for every $\alpha\in w\Delta^+_\mathfrak{q}$:
	\begin{align*}
	\pi(e_\alpha)v_{w\lambda}&=0,\\
	\gamma(e_\alpha)u_{w\rho-\rho_{\mathfrak{h}}}&=0.\end{align*}
	Arguing as above, one deduces that $F_{w(\lambda+\rho)-\rho_\mathfrak{h}}$ is contained in $\ker\widehat{D}_{\mathfrak{g},\mathfrak{h}}(V)$. Since $w\in W^1$ was arbitrary, we conclude that $\ker D_{\mathfrak{g},\mathfrak{h}}(V)\subseteq\ker\widehat{D}_{\mathfrak{g},\mathfrak{h}}(V)$.
\end{proof}

\subsection{Weyl's inequalities}\label{weylssubsection}
The following inequalities, known as Weyl's inequalities, turn out to be essential for the study of $\ker D^t_{\mathfrak{g},\mathfrak{h}}(V)$.

\begin{proposition}[Weyl's inequalities]
	Let $A$ and $B$ be two Hermitian operators acting on a $n$-dimensional Hilbert space $\mathcal{H}$. Let $\lambda_i(A)$, $1\leq i\leq n$, (respectively $\lambda_i(B)$ and $\lambda_i(A+B)$) be the (real) eigenvalues of $A$ (respectively $B$ and $A+B$) in descending ordering counting multiplicities. Then, one has:
	
	\begin{equation}\label{wineq}
	\lambda_{i+j-1}(A+B)\leq\lambda_i(A)+\lambda_j(B),
	\end{equation}
	whenever the indices make sense.
\end{proposition}

A proof of Weyl's inequalities is based on min-max principle and can be found in \cite{brion,tao}.

\subsection{Related weights}
Consider the decomposition 
\begin{equation*}
V\otimes S=\bigoplus_{\mu\in\mathfrak{t}^*}(V\otimes S)_\mu
\end{equation*}
of $V\otimes S$ into isotypic components. From the $\mathfrak{h}$-invariance of $D^t_{\mathfrak{g},\mathfrak{h}}(V)$, each isotypic component $(V\otimes S)_\mu$ is stable under the action of $D^t_{\mathfrak{g},\mathfrak{h}}(V)$. Therefore, in order to calculate $\ker D^t_{\mathfrak{g},\mathfrak{h}}(V)$, it suffices to examine how each isotypic component can contribute to $\ker D^t_{\mathfrak{g},\mathfrak{h}}(V)$. 

Fix an isotypic component $(V\otimes S)_\mu$ and let $\lambda_i^{(\mu)}(D^t_{\mathfrak{g},\mathfrak{h}}(V))$ be the (real) eigenvalues in descending ordering and counting multiplicities of $D^t_{\mathfrak{g},\mathfrak{h}}(V)$ restricted to $(V\otimes S)_\mu$ and $\lVert \sqrt{2}\otimes\gamma(c)\rVert^{(\mu)}$ the operator-norm of $\sqrt{2}\otimes \gamma(c)$ on $(V\otimes S)_\mu$. Then, taking $A=D_{\mathfrak{g},\mathfrak{h}}(V)$, $B=(1-t)\big(\sqrt{2}\otimes \gamma(c)\big)$ and $j=1$, the inequality \eqref{wineq} becomes
\begin{subequations}\label{blockineq}
	\begin{equation}
	\lambda_i^{(\mu)}(D^t_{\mathfrak{g},\mathfrak{h}}(V))\leq\lambda_i^{(\mu)}(D_{\mathfrak{g},\mathfrak{h}}(V))+\lvert1-t\rvert\cdot\lVert \sqrt{2}\otimes\gamma(c)\rVert^{(\mu)}
	\end{equation}
	and, exchanging the roles of $D_{\mathfrak{g},\mathfrak{h}}(V)$ and $(1-t)\big(\sqrt{2}\otimes \gamma(c)\big)$: 
	\begin{equation}\label{otherineq}
	\lambda_i^{(\mu)}(D_{\mathfrak{g},\mathfrak{h}}(V))-\lvert1-t\rvert\cdot\lVert \sqrt{2}\otimes\gamma(c)\rVert^{(\mu)}\leq\lambda_i^{(\mu)}(D^t_{\mathfrak{g},\mathfrak{h}}(V)).
	\end{equation}
\end{subequations}

The operator $D_{\mathfrak{g},\mathfrak{h}}(V)$ acts on $(V\otimes S)_\mu$ and is diagonalizable while $D_{\mathfrak{g},\mathfrak{h}}(V)^2$ acts by the scalar 
\begin{equation*}
\lVert\lambda+\rho\rVert^2-\lVert\mu+\rho_\mathfrak{h}\rVert^2.
\end{equation*}
Hence, the eigenvalues $\lambda_i(D_{\mathfrak{g},\mathfrak{h}}(V))$ of $D_{\mathfrak{g},\mathfrak{h}}(V)$ restricted to $(V\otimes S)_\mu$  are
\begin{equation*}
\pm\sqrt{\lVert\lambda+\rho\rVert^2-\lVert\mu+\rho_\mathfrak{h}\rVert^2}.
\end{equation*}

Let us see what happens with the eigenvalues of $\sqrt{2}\otimes\gamma(c)$ on such an isotypic component. One checks that the actions of $\sqrt{2}\otimes\gamma(c)$ and $-D_{\mathfrak{g},\mathfrak{h}}(V_0)$ coincide if $V_0$ is the trivial representation of $\mathfrak{g}$. Consequently, $\big(\sqrt{2}\otimes\gamma(c)\big)^2$ acts by the scalar 
\begin{equation*}
\lVert\rho\rVert^2-\lVert\mu_1+\rho_\mathfrak{h}\rVert^2
\end{equation*}
on the $\mu_1$-isotypic component $(V_0\otimes S)_{\mu_1}$ of $V_0\otimes S$ and the eigenvalues of $\sqrt{2}\otimes \gamma(c)$ in $V_0\otimes S_{\mu_1}$ are exactly 
\begin{equation*}\pm\sqrt{\lVert\rho\rVert^2-\lVert\mu_1+\rho_\mathfrak{h}\rVert^2}.\end{equation*} 
Therefore the eigenvalues of $\sqrt{2}\otimes\gamma(c)$ on $(V\otimes S)_\mu$ are of this form, when $\mu_1$ is a weight of $S$ such that $S_{\mu_1}$ is involved in $(V\otimes S)_\mu$. In other words, $\mu_1$ is such that $\mu-\mu_1$ is a weight of $V$. 

\begin{definition}[Related weight]\label{related}
	When the weight $\mu_1\in\mathfrak{t}^*$ of $S$ is such that $\mu-\mu_1$ is a weight of $\hspace{1mm}V$, we say that $\mu_1$ is related to $\mu$.
\end{definition}

\begin{theorem}\label{inequal}
	If $\mu\in\mathfrak{t}^*$ is a weight of $\hspace{1mm}V\otimes S$ and $\mu_1\in\mathfrak{t}^*$ is related to $\mu$, then 
	
	\begin{equation}\label{ineq}
	\lVert\lambda+\rho\rVert^2-\lVert\mu+\rho_\mathfrak{h}\rVert^2\geq\lVert\rho\rVert^2-\lVert\mu_1+\rho_\mathfrak{h}\rVert^2.
	\end{equation}
	The equality holds if and only if $\mu-\mu_1$ is an extremal weight of $\hspace{1mm}V$ and $\rho-w(\mu_1+\rho_\mathfrak{h})$ is orthogonal to $\lambda$, with $w\in W$ being such that $w(\mu-\mu_1)$ is $\Delta^+$-dominant. In this case $\rho-w(\mu_1+\rho_\mathfrak{h})$ can be written as a sum $\sum\limits_{\alpha\in A}\alpha$, for some $A\subseteq \Delta^+$ and $\langle\alpha,\lambda\rangle=0$ for every $\alpha$ in $A$.
	
\end{theorem}

Before we prove this statement, we will first prove the following lemma.

\begin{lemma}\label{otherspin}
	Let $B$ be a subset of $\Delta^+$ and $w$ be an element of the Weyl group $W$. Then there is a subset $A$ of $\Delta^+$ such that
	
	\begin{equation*}
	w(\rho-\sum\limits_{\beta\in B}\beta)=\rho-\sum\limits_{\alpha\in A}\alpha.
	\end{equation*}	
\end{lemma}

\begin{proof}[Proof of Lemma \ref{otherspin}]
	Let $\mathfrak{h}':=\mathfrak{t}$ and $\mathfrak{q}':=(\mathfrak{h}')^\perp$. Then $\mathfrak{q}'=\mathfrak{n}\oplus\mathfrak{n}^-$, where
	\begin{equation*} 
	\mathfrak{n}=\bigoplus\limits_{\alpha\in\Delta^+}\mathfrak{g}_\alpha,\quad \mathfrak{n}^-=\bigoplus\limits_{\alpha\in\Delta^+}\mathfrak{g}_{-\alpha}.
	\end{equation*}
	Let $S':=\bigwedge\mathfrak{n}^-$ be a space of spinors for $\mathbf{C}(\mathfrak{q}')$. 
	Moreover, for $w\in W$, let $\Delta^+_w:=w\Delta^+$ and $\mathfrak{n}_w$ (respectively $\mathfrak{n}^-_w$) be the direct sum of positive (respectively negative) root spaces with respect to the positive system $\Delta^+_w$. Let $S_w':=\bigwedge\mathfrak{n}^-_w$ be the space of spinors of $\mathbf{C}(\mathfrak{q}')$. Then $w\rho$ is the half sum of positive roots of $\Delta^+_w$,  $w\rho_{\mathfrak{h}'}=0$ and the element 
	\begin{equation*}
	w\rho-\sum\limits_{\beta\in B}w\beta=w(\rho-\sum\limits_{\beta\in B}\beta),
	\end{equation*} 
	being of the form \eqref{Sweight}, is a weight of $S'_w$. On the other hand, $S_w'$ and $S'$ are isomorphic as $\mathbf{C}(\mathfrak{q}')$-modules \cite[Theorem 6.1.3]{goodman} and thus they have the same weights.  As a consequence, the element 
	\begin{equation*}
	w(\rho-\sum\limits_{\beta\in B}\beta)
	\end{equation*}
	is a weight of $S'$ and so of the form \eqref{Sweight}. In other words, there is a subset $A$ of $\Delta^+$ such that 
	
	\begin{equation*}
	w(\rho-\sum\limits_{\beta\in B}\beta)=\rho-\sum\limits_{\alpha\in A}\alpha.\qedhere
	\end{equation*}	
\end{proof}

\begin{proof}[Proof of Theorem \ref{inequal}] 
	Let $\mu\in\mathfrak{t}^*$ be a weight of $\hspace{1mm}V\otimes S$ and $\mu_1\in\mathfrak{t}^*$ be a weight of $S$ related to $\mu$. Then
	
	\begin{equation*}
	\begin{array}{crl}
	&\lVert\lambda+\rho\rVert^2-\lVert\mu+\rho_\mathfrak{h}\rVert^2&\geq\lVert\rho\rVert^2-\lVert\mu_1+\rho_\mathfrak{h}\rVert^2\\
	\Leftrightarrow&\lVert\lambda+\rho\rVert^2-\lVert\rho\rVert^2&\geq\lVert\mu+\rho_\mathfrak{h}\rVert^2-\lVert\mu_1+\rho_\mathfrak{h}\rVert^2\\
	\Leftrightarrow&\lVert\lambda\rVert^2+ 2\langle\lambda,\rho\rangle&\geq\lVert\mu-\mu_1\rVert^2+2\langle\mu-\mu_1,\mu_1+\rho_\mathfrak{h}\rangle\\
	\Leftrightarrow&(\lVert\lambda\rVert^2-\lVert\mu-\mu_1\rVert^2)+ 2\langle\lambda,\rho\rangle&\geq2\langle\mu-\mu_1,\mu_1+\rho_\mathfrak{h}\rangle.
	\end{array}
	\end{equation*}
	The weight $\mu_1$ is of the form \eqref{Sweight}, i.e. there is a subset $B$ of $\Delta^+$ (more precisely of $\Delta^+\setminus\Delta^+_\mathfrak{h}$) such that 
	\begin{equation*}
	\mu_1=\rho-\rho_\mathfrak{h}-\sum\limits_{\beta\in B}\beta
	\end{equation*}
	and thus
	\begin{equation*}
	\mu_1+\rho_\mathfrak{h}=\rho-\sum\limits_{\beta\in B}\beta.
	\end{equation*}
	Let $w$ be an element of  $W$ such that $w(\mu-\mu_1)$ is dominant. Using Lemma \ref{otherspin}, one can find a subset $A$ of $\Delta^+$ such that
	\begin{equation*}
	w(\mu_1+\rho_\mathfrak{h})=\rho-\sum\limits_{\alpha\in A}\alpha.\end{equation*}
	Consequently
	\begin{align*}
	\langle\mu-\mu_1,\mu_1+\rho_\mathfrak{h}\rangle&=\langle w(\mu-\mu_1),w(\mu_1+\rho_\mathfrak{h})\rangle\\
	&=\langle w(\mu-\mu_1),\rho-\sum\limits_{\alpha\in A}\alpha\rangle\\
	&=\langle w(\mu-\mu_1),\rho\rangle-\sum\limits_{\alpha\in A}\langle w(\mu-\mu_1),\alpha\rangle.\\
	\end{align*}
	Replacing in the above inequality we obtain
	
	\begin{equation*}
	\begin{array}{crl}
	&(\lVert\lambda\rVert^2-\lVert\mu-\mu_1\rVert^2)+ 2\langle\lambda,\rho\rangle\geq&\hspace{-2mm}2\langle\mu-\mu_1,\mu_1+\rho_\mathfrak{h}\rangle\\
	\Leftrightarrow&(\lVert\lambda\rVert^2-\lVert\mu-\mu_1\rVert^2)+ 2\langle\lambda,\rho\rangle\geq&\hspace{-2mm}2\langle w(\mu-\mu_1),\rho\rangle\\
	&&-2\sum\limits_{\alpha\in A}\langle w(\mu-\mu_1),\alpha\rangle\\
	\Leftrightarrow&(\lVert\lambda\rVert^2-\lVert\mu-\mu_1\rVert^2)+2\langle\lambda-w(\mu-\mu_1),\rho\rangle\geq&\hspace{-2mm} -2\sum\limits_{\alpha\in A}\langle w(\mu-\mu_1),\alpha\rangle,
	\end{array}
	\end{equation*}
	which is always true. More precisely, since $\mu-\mu_1$ is a weight of $V$ and $\lambda$ is extremal, we always have $\lVert\lambda\rVert^2\geq\lVert\mu-\mu_1\rVert^2$. In addition, $w(\mu-\mu_1)$ is a weight of $V$ and so $\lambda
	-w(\mu-\mu_1)$ can be written as a positive sum of positive roots. The element $\rho$ being dominant \cite[Proposition 2.67]{Knapp1}, the left-hand side of the last inequality is always nonnegative. On the other hand, $w(\mu-\mu_1)$ is also dominant and so the right-hand side is always nonpositive.
	
	The equality holds if and only if 
	
	\begin{equation*}
	(\lVert\lambda\rVert^2-\lVert\mu-\mu_1\rVert^2)+2\langle\lambda-w(\mu-\mu_1),\rho\rangle=0,
	\end{equation*}
	which means that $\lambda=w(\mu-\mu_1)$ and thus $\mu-\mu_1$ is extremal, and
	
	\begin{equation*}
	\langle w(\mu-\mu_1),\alpha\rangle=0 \text{ for every $\alpha$ in $A$,}
	\end{equation*}
	which is equivalent to $\langle\lambda,\alpha\rangle=0$ for every $\alpha$ in $A$.
\end{proof}

\subsection{Kernel of $t$-noncubic Dirac operators}
We keep the previous notation. Namely, $\mathfrak{g}$ is a complex semisimple Lie algebra, $\mathfrak{h}$ a Lie subalgebra of $\mathfrak{g}$ satisfying condition \eqref{conditionh} with $\mathfrak{t}$ being a common Cartan subalgebra of $\mathfrak{g}$ and $\mathfrak{h}$, while $\mathfrak{q}$ is the orthogonal complement of $\mathfrak{h}$ with respect to the Killing form $\langle\cdot,\cdot\rangle$ of $\mathfrak{g}$.

\begin{theorem}\label{kerneltcubic}
	Let $V$ be an irreducible finite-dimensional representation of $\mathfrak{g}$ of highest weight $\lambda\in\mathfrak{t}^*$. For every $t\in(0,2)$,
	\begin{equation*}
	\ker D^t_{\mathfrak{g},\mathfrak{h}}(V)=\ker D_{\mathfrak{g},\mathfrak{h}}(V).
	\end{equation*}
\end{theorem}

\begin{proof} The inclusion $\ker D_{\mathfrak{g},\mathfrak{h}}(V)\subseteq\ker D^t_{\mathfrak{g},\mathfrak{h}}(V)$ for every $t\in(0,2)$ is ensured by Proposition \ref{kerinclusion}.
	
	For the other inclusion, we suppose that $(V\otimes S)_\mu\cap\ker D_{\mathfrak{g},\mathfrak{h}}(V)\neq 0$, i.e. there is some index $i$ such that $\lambda_i^{(\mu)}(D^t_{\mathfrak{g},\mathfrak{h}}(V))=0$. Suppose that the corresponding eigenvalue $\lambda_i^{(\mu)}(D_{\mathfrak{g},\mathfrak{h}}(V))$ of the cubic Dirac operator $D_{\mathfrak{g},\mathfrak{h}}(V)$ is nonnegative, i.e.
	
	\begin{equation*}
	\lambda_i^{(\mu)}(D_{\mathfrak{g},\mathfrak{h}}(V))=\sqrt{\lVert\lambda+\rho\rVert^2-\lVert\mu+\rho_\mathfrak{h}\rVert^2}.
	\end{equation*}
	On the other hand, recall that 
	\begin{equation*}
	\lVert\sqrt{2}\otimes\gamma(c)\rVert^{(\mu)}=\sqrt{\lVert\rho\rVert^2-\lVert\mu_1+\rho_\mathfrak{h}\rVert^2},
	\end{equation*}
	where $\mu_1\in\mathfrak{t}^*$ is a weight of $S$ related to $\mu$ (see Definition \ref{related}). 
	Replacing in \eqref{otherineq} and using \eqref{ineq}, one obtains
	\begin{align*}
	0=\lambda_i^{(\mu)}(D^t_{\mathfrak{g},\mathfrak{h}}(V))&\geq\sqrt{\lVert\lambda+\rho\rVert^2-\lVert\mu+\rho_\mathfrak{h}\rVert^2}-\lvert1-t\rvert\sqrt{\lVert\rho\rVert^2-\lVert\mu_1+\rho_\mathfrak{h}\rVert^2}\\
	&\geq\sqrt{\lVert\lambda+\rho\rVert^2-\lVert\mu+\rho_\mathfrak{h}\rVert^2}-\sqrt{\lVert\rho\rVert^2-\lVert\mu_1+\rho_\mathfrak{h}\rVert^2}\\
	&\geq0.
	\end{align*}
	
	Hence the equalities in the above inequalities must hold and one obtains
	\begin{subequations}\label{iinn}
		\begin{equation}
		\sqrt{\lVert\lambda+\rho\rVert^2-\lVert\mu+\rho_\mathfrak{h}\rVert^2}-\lvert1-t\rvert\sqrt{\lVert\rho\rVert^2-\lVert\mu_1+\rho_\mathfrak{h}\rVert^2}=0,
		\end{equation}
		\begin{equation}
		\sqrt{\lVert\lambda+\rho\rVert^2-\lVert\mu+\rho_\mathfrak{h}\rVert^2}-\sqrt{\lVert\rho\rVert^2-\lVert\mu_1+\rho_\mathfrak{h}\rVert^2}=0.
		\end{equation}
	\end{subequations}
	Since $\lvert1-t\rvert\neq1$, subtracting the above equalities gives 
	\begin{equation}\label{innn}
	\sqrt{\lVert\lambda+\rho\rVert^2-\lVert\mu+\rho_\mathfrak{h}\rVert^2}=\sqrt{\lVert\rho\rVert^2-\lVert\mu_1+\rho_\mathfrak{h}\rVert^2}=0
	\end{equation}
	and so $\lambda_i^{(\mu)}(D_{\mathfrak{g},\mathfrak{h}}(V))=0$.
	
	A similar argument applies in the case when $\lambda_i^{(\mu)}(D_{\mathfrak{g},\mathfrak{h}}(V))$ is nonpositive.
	In other words, the eigenvalues of $D_{\mathfrak{g},\mathfrak{h}}(V)$ in $(V\otimes S)_\mu$ all are zero and so $(V\otimes S)_\mu\subseteq\ker D_{\mathfrak{g},\mathfrak{h}}(V)$. 
	Consequently, we deduce that  if $(V\otimes S)_\mu\cap\ker D^t_{\mathfrak{g},\mathfrak{h}}(V)\neq 0$, then $(V\otimes S)_\mu\subseteq \ker D_{\mathfrak{g},\mathfrak{h}}(V)$ and so $\ker D^t_{\mathfrak{g},\mathfrak{h}}(V)\subseteq \ker D_{\mathfrak{g},\mathfrak{h}}(V)$.	
\end{proof}

\begin{remark}\label{remak}
	Note that in the case when $t=0$ or $t=2$ the equalities \eqref{iinn} coincide so that the equality \eqref{innn} can not be deduced. Nevertheless, one has that the equality in \eqref{ineq} must hold.
\end{remark}

\subsection{Kernel of the noncubic Dirac operator}
We address the case of the noncubic Dirac operator $\widehat{D}_{\mathfrak{g},\mathfrak{h}}(V)$. The same arguments can apply exactly in the same way for the operator $D^2_{\mathfrak{g},\mathfrak{h}}(V)$.

\begin{definition} We say that the equality in \eqref{ineq} holds for a weight $\mu\in\mathfrak{t}^*$ of $\hspace{1mm}V\otimes S$ (or equivalently for the corresponding isotypic component $(V\otimes S)_\mu$), if there is $\mu_1\in\mathfrak{t}^*$ related to $\mu$ such that the equality in \eqref{ineq} holds. In this case, we will say that $\mu_1$ is strongly related to $\mu$.
\end{definition}

The following theorem establishes a necessary condition for the isotypic component $(V\otimes S)_\mu$ to contribute to $\ker\widehat{D}_{\mathfrak{g},\mathfrak{h}}(V)$.

\begin{theorem}\label{cont} Let $\mu\in\mathfrak{t}^*$ be a weight of $\text{ }V\otimes S$. If the equality in \eqref{ineq}  does not hold for $\mu$, then $(V\otimes S)_\mu$ does not contribute to $\ker\widehat{D}_{\mathfrak{g},\mathfrak{h}}(V)$. In other words, the isotypic component $(V\otimes S)_\mu$ has nontrivial intersection with $\ker\widehat{D}_{\mathfrak{g},\mathfrak{h}}(V)$ only if the equality in $\eqref{ineq}$ holds for $(V\otimes S)_\mu$.
\end{theorem}




\begin{proof}
	If $\mu\in\mathfrak{t}^*$ is a weight of $V\otimes S$ such that $(V\otimes S)_\mu$ contributes to $\ker \widehat{D}_{\mathfrak{g},\mathfrak{h}}(V)$, i.e.  for some index $i$ there is a zero eigenvalue $\lambda_i^{(\mu)}(\widehat{D}_{\mathfrak{g},\mathfrak{h}}(V))$, then, according to Remark \ref{remak}, the equality in \eqref{ineq} must hold.
\end{proof}

Consequently, in order to study $\ker\widehat{D}_{\mathfrak{g},\mathfrak{h}}(V)$ we can focus on the isotypic components for which the equality \eqref{ineq} holds. 


In the rest of this section, we suppose that $\mathfrak{h}$ is a Cartan subalgebra $\mathfrak{t}$ of $\mathfrak{g}$. Under this asumption, we give a sufficient condition for a $\mathfrak{t}$-isotypic component $(V\otimes S)_\mu$ of $V\otimes S$ to contribute to $\ker\widehat{D}_{\mathfrak{g},\mathfrak{t}}(V)$. Combining this condition with Theorem \ref{cont} we obtain a complete description for $\ker\widehat{D}_{\mathfrak{g},\mathfrak{t}}(V)$ in Theorem \ref{thmbasic}.

If $\mu_1\in \mathfrak{t}^*$ is strongly related to $\mu$, then it will also be the case for $w\mu$ and $w\mu_1$ for every $w$ in $W$. Note that $w\mu$ and $w\mu_1$ turn out to be weights of $V\otimes S$ and $S$ respectively. More precisely, $w\mu_1$ is a weight of $S$ since it is of the form \eqref{Sweight} with $\rho_\mathfrak{h}=0$ and with $w\Delta^+$ being the positive system of $\Delta$. Then $w\mu=w(\mu-\mu_1)+w\mu_1$ is a weight of $V\otimes S$ since $\mu-\mu_1$, and thus $w(\mu-\mu_1)$, is a weight of $V$.

Set 
\begin{equation*}
(V\otimes S)^{[\mu]}:=\bigoplus\limits_{w\in W/W_\mu}(V\otimes S)_{w\mu},
\end{equation*}
where $W_\mu$ stands for the stabilizer of $\mu$ in $W$. In other words, $(V\otimes S)^{[\mu]}$ is the direct sum of the isotypic components corresponding to the $W$-orbit of $\mu$.
%
%
From the previous discussion, if the equality in \eqref{ineq} holds for $\mu$, then it will also hold for any isotypic component appearing in $(V\otimes S)^{[\mu]}$ and so, using the action of $W/W_\mu$, we can suppose that $\mu$ is such that if $\mu_1$ is strongly related to $\mu$, then $\mu-\mu_1$ is dominant. In this case, according to Theorem \ref{inequal}, $\mu-\mu_1$ is an extremal weight of $V$ and, since it is dominant, it must be equal to $\lambda$. Then $\mu_1$ can be written as \begin{equation*}\mu_1=\rho-\sum\limits_{\alpha\in A}\alpha
\end{equation*}
for some $A\subseteq\big(\mathbb{R}\{\lambda\}\big)^\perp\cap\Delta^+$ and so
\begin{equation*}
\mu=\lambda+\rho-\sum\limits_{\alpha\in A}\alpha.
\end{equation*}
Consequently, if
\begin{equation*}
\widetilde{A}(\lambda):=\{\lambda+\rho-\sum\limits_{\alpha\in A}\alpha\mid A\subseteq \big(\mathbb{R}\{\lambda\}\big)^\perp\cap\Delta^+\}
\end{equation*} 
and 
\begin{equation}\label{alphalambda}
A(\lambda):=\widetilde{A}(\lambda)/W,\end{equation}
i.e. $A(\lambda)$ is the set of equivalence classes of $\widetilde{A}(\lambda)$ with respect to $W$,
the equality in \eqref{ineq} holds exactly for the isotypic components appearing in 

\begin{equation*}
\bigoplus\limits_{\mu\in A(\lambda)}(V\otimes S)^{[\mu]}.
\end{equation*}

Note that for distinct elements $\mu$ and $\mu'$ in $\widetilde{A}(\lambda)$ belonging to the same $W$-orbit, we have $(V\otimes S)^{[\mu]}=(V\otimes S)^{[\mu']}$. In order to avoid having twice the same block in the above direct sum, we have to consider the quotient $A(\lambda)$ instead of $\widetilde{A}(\lambda)$.

\begin{lemma}\label{onevector} If the equality in \eqref{ineq} holds for $\mu$, then there is a weight $\mu_0$ of $V$ and a weight $\mu_1$ of $S$ (related to $\mu_0$) such that $(V\otimes S)_\mu=V_{\mu_0}\otimes S_{\mu_1}$ where $V_{\mu_0}$ (respectively $S_{\mu_1}$) is the corresponding weight subspace of $V$ (respectively $S$). In other words, there is only one weight $\mu_1$ of $S$ related to $\mu$.\end{lemma}

\begin{proof}
	Assume that $(V\otimes S)_\mu$ contains a nontrivial vector of the form $v_\lambda\otimes e$, with $v_\lambda$ being a highest weight vector in $V$ and $e$ being a weight vector of $S$. Suppose that there is a nontrivial weight vector $v_{\lambda'}$ of weight $\lambda'$ in $V$ and a nontrivial weight vector $e'$ in $S$ such that $v_{\lambda'}\otimes e'$ also belongs to $(V\otimes S)_\mu$. Let $A$ and $A'$ be subspaces of $\Delta^+$ such that the weights of $S$ corresponding to $e$ and $e'$ are respectively
	\begin{equation*}
	\rho-\sum\limits_{\alpha\in A}\alpha\mathrm{ \quad and \quad} \rho-\sum\limits_{\alpha'\in A'}\alpha'.
	\end{equation*}
	We observe that, according to Theorem \ref{inequal}, $\langle\alpha,\lambda\rangle=0$ for every $\alpha$ in $A$. Therefore
	\begin{equation*}
	\begin{array}{crl}
	&\lambda+\rho-\sum\limits_{\alpha\in A}\alpha&=\lambda'+\rho-\sum\limits_{\alpha'\in A'}\alpha'\\
	\Rightarrow&\lambda-\lambda'&=\sum\limits_{\alpha\in A}\alpha-\sum\limits_{\alpha'\in A'}\alpha'\\
	\Rightarrow&\langle\lambda-\lambda',\lambda\rangle&=-\sum\limits_{\alpha'\in A'}\langle\alpha',\lambda\rangle.\\
	\end{array}
	\end{equation*}
	The left side of the above equality is nonnegative while the right one is nonpositive and thus both must be equal to $0$. Therefore $\langle\lambda-\lambda',\lambda\rangle=0$ and 
	\begin{align*}
	\lVert\lambda'\rVert^2&=\lVert(\lambda'-\lambda)+\lambda\rVert^2\\
	&=\Vert\lambda'-\lambda\rVert^2+\lVert\lambda\rVert
	^2+2\langle\lambda'-\lambda,\lambda\rangle\\
	&=\Vert\lambda'-\lambda\rVert^2+\lVert\lambda\rVert
	^2\\
	&\geq \lVert\lambda\rVert^2.
	\end{align*}
	Since $\lambda$ is an extremal weight of $V$, the equality must hold. Therefore $\Vert\lambda'-\lambda\rVert=0$ or equivalently $\lambda'=\lambda$.
\end{proof}

\begin{theorem}\label{thmbasic} Let $\mathfrak{g}$ be a complex semisimple Lie algebra, $\mathfrak{h}:=\mathfrak{t}$ be a Cartan subalgebra of $\mathfrak{g}$, $\Delta$ be the corresponding root system and fix $\Delta^+$ a positive system of $\Delta$. For a $\Delta^+$-dominant algebraically integral element $\lambda\in\mathfrak{t}^*$, let $V$ be the corresponding finite-dimensional irreducible representation of highest weight $\lambda$ and $A(\lambda)$ be the set defined in \eqref{alphalambda}. Then the kernel of the noncubic Dirac operator $\widehat{D}_{\mathfrak{g},\mathfrak{t}}(V)$ acting on $V\otimes S$ is
	
	\begin{equation}\label{descriptionker}
	\ker\widehat{D}_{\mathfrak{g},\mathfrak{t}}(V)=\bigoplus\limits_{\mu\in A(\lambda)}(V\otimes S)^{[\mu]}.
	\end{equation}
\end{theorem}

\begin{proof} In view of Theorem \ref{cont}, it suffices to prove that each isotypic component of the sum on the right-hand side of the above relation lies in $\ker\widehat{D}_{\mathfrak{g},\mathfrak{t}}(V)$.
	Let $(V\otimes S)_\mu$ be such a component, i.e. $\mu$ is a weight of $V\otimes S$ such that \eqref{ineq} holds as an equality. According to Lemma \ref{onevector}, there is a weight $\mu_0$ of $V$ such that $(V\otimes S)_\mu\subseteq V_{\mu_0}\otimes S$, where $V_{\mu_0}$ is the weight subspace of $V$ corresponding to $\mu_0$. Then 
	\begin{equation}\label{last1}
	\widehat{D}_{\mathfrak{g},\mathfrak{t}}(V)[(V\otimes S)_\mu]\subseteq (V\otimes S)_\mu\subseteq V_{\mu_0}\otimes S.
	\end{equation}
	On the other hand, each additive term of $\widehat{D}_{\mathfrak{g},\mathfrak{t}}(V)$ acts on an element of $V_{\mu_0}\otimes S$ giving an element in $V_{\mu_0+\alpha}\otimes S$ for some $\alpha$ in $\Delta$, where $V_{\mu_0+\alpha}$ is the weight subspace of $V$ corresponding to $\mu_0+\alpha$ (which may be trivial if $\mu_0+\alpha$ is not a weight of $V$). Thus
	\begin{equation*}
	\widehat{D}_{\mathfrak{g},\mathfrak{t}}(V)[(V\otimes S)_\mu]\subseteq \widehat{D}_{\mathfrak{g},\mathfrak{t}}(V)[V_{\mu_0}\otimes S]\subseteq\bigoplus\limits_{\alpha\in\Delta}V_{\mu_0+\alpha}\otimes S
	\end{equation*}
	so that
	\begin{equation*}
	\widehat{D}_{\mathfrak{g},\mathfrak{t}}(V)[(V\otimes S)_\mu]\subseteq\Big( \bigoplus\limits_{\alpha\in\Delta}V_{\mu_0+\alpha}\otimes S\Big)\cap(V_{\mu_0}\otimes S)
	\end{equation*}
	Since $V_{\mu_0}$ is linearly independent with $\bigoplus\limits_{\alpha\in\Delta}V_{\mu_0+\alpha}$ we deduce that 
	\begin{equation*}
	\widehat{D}_{\mathfrak{g},\mathfrak{t}}(V)[(V\otimes S)_\mu]=\{0\}
	\end{equation*}
	and so $(V\otimes S)_\mu$ lies in $\ker\widehat{D}_{\mathfrak{g},\mathfrak{h}}(V)$.
\end{proof}
	

The following proposition is a consequence of Theorem \ref{thmbasic}.

\begin{proposition}\label{intersectionD}
	For $\{e_\alpha \}_{\alpha\in \Delta}$ as above,
	\begin{equation*}
	\ker\widehat{D}_{\mathfrak{g},\mathfrak{h}}(V)=\bigcap\limits_{\alpha\in\Delta}\ker\big(\pi(e_\alpha)\otimes \gamma(e_{-\alpha})\big).
	\end{equation*}
	In other words, if an element $x$ of $V\otimes S$ is annihilated by $\widehat{D}_{\mathfrak{g},\mathfrak{h}}(V)$, then it is annihilated by every additive term $\pi(e_\alpha)\otimes \gamma(e_{-\alpha})$, $\alpha\in\Delta$, of $\widehat{D}_{\mathfrak{g},\mathfrak{h}}(V)$.
\end{proposition}

\begin{proof}
	One inclusion is obvious. For the other one, if $x\in(V\otimes S)_\mu$ belongs to $\ker \widehat{D}_{\mathfrak{g},\mathfrak{t}}(V)$, then the equality in $\eqref{ineq}$ holds for $(V\otimes S)_\mu$. Repeating the argument used in the proof of Theorem \ref{thmbasic}, for $\pi(e_\alpha)\otimes \gamma(e_{-\alpha})$, $\alpha\in\Delta$, instead of $\widehat{D}_{\mathfrak{g},\mathfrak{t}}(V)$, one obtains that
	\begin{equation*}
	\pi(e_\alpha)\otimes \gamma(e_{-\alpha})[(V\otimes S)_\mu]\subseteq( 
	V_{\mu_0+\alpha}\otimes S)\cap(V_{\mu_0}\otimes S)=\{0\}
	\end{equation*}
	and so 
	$\ker \widehat{D}_{\mathfrak{g},\mathfrak{t}}(V)\subseteq \ker \big(\pi(e_\alpha)\otimes \gamma(e_{-\alpha})\big)$. Since $\alpha\in\Delta$ was arbitrary, we obtain the desired result.
\end{proof}

\subsection{Explicit description of the set $\boldsymbol{A(\lambda)}$}\label{thesetA}
In order to make more explicit our description \eqref{descriptionker} of the kernel of the noncubic Dirac operator $\widehat{D}_{\mathfrak{g},\mathfrak{h}}(V)$, we study the set $A(\lambda)$ involved in the formula.
We continue here with the assumption that $\mathfrak{h}$ coincides with a Cartan subalgebra $\mathfrak{t}$ of $\mathfrak{g}$. The main result of this paragraph is the following proposition.

\begin{proposition}\label{basic}
	The set $\{\mu\in\widetilde{A}(\lambda)\mid \mu-\lambda \text{ is }\Delta^+\text{-dominant}\}$ is a complete family of representatives for $A(\lambda)$.
\end{proposition}
In other words, by taking all elements 

\begin{equation*}
\lambda+\rho-\sum\limits_{\alpha\in A}\alpha\in\widetilde{A}(\lambda),\end{equation*}
with $\rho-\sum\limits_{\alpha\in A}\alpha$
being $\Delta^+$-dominant, we obtain a complete family of representatives for $A(\lambda)$. Recall that 
\begin{equation*}
\lambda+\rho-\sum\limits_{\alpha\in A}\alpha \end{equation*} 
belongs to $\widetilde{A}(\lambda)$ if $A$ is a subset of $\big(\mathbb{R}\{\lambda\}\big)^{\perp}\cap\Delta^+$. The rest of this section is devoted to the proof of Proposition \ref{basic}.

\begin{lemma} Let $A$ and $A'$ be subsets of $\big(\mathbb{R}\{\lambda\}\big)^{\perp}\cap\Delta^+$ and set
	
	\begin{equation*}
	\begin{array}{cl}\mu&=\lambda+\rho-\sum\limits_{\alpha\in A}\alpha,\\
	\mu'&=\lambda+\rho-\sum\limits_{\alpha'\in A'}\alpha'
	\end{array}
	\end{equation*}
	to be the corresponding elements of $\widetilde{A}(\lambda)$. Then $\mu$ and $\mu'$ are $W$-conjugate if and only if they are $W_\lambda$-conjugate, where $W_\lambda$ stands for the stabilizer of $\lambda$ in $W$.
\end{lemma}

\begin{proof}
	The "if" direction is obvious. For the "only if" direction, let us suppose that $w$ is an element of $W$ such that $\mu=w\mu'$. Then 
	\begin{equation*}
	\begin{array}{crl}
	&\mu&=w\mu'\\
	\Longrightarrow&\lambda+\rho-\sum\limits_{\alpha\in A}\alpha&=w(\lambda+\rho-\sum\limits_{\alpha'\in A'}\alpha')\\
	\Longrightarrow&\lambda-w\lambda&=w(\rho-\sum\limits_{\alpha'\in A'}\alpha')-(\rho-\sum\limits_{\alpha\in A}\alpha)\\
	\Longrightarrow&\langle\lambda-w\lambda,\lambda\rangle&=-\langle(\rho-\sum\limits_{\alpha\in A}\alpha)-w(\rho-\sum\limits_{\alpha'\in A'}\alpha'),\lambda\rangle\\
	\Longrightarrow&\langle\lambda-w\lambda,\lambda\rangle&=-\langle\rho-w(\rho-\sum\limits_{\alpha'\in A'}\alpha'),\lambda\rangle.
	\end{array}
	\end{equation*}
	Since $w\lambda$ 
	is a weight of $V$, $\lambda-w\lambda$ is a nonnegative sum of positive roots, while $\lambda$ is dominant. Thus, the left-hand side of the above expression is nonnegative. On the other hand, using Lemma \ref{otherspin}, one can see that the right-hand side is nonpositive and so both sides must be $0$. Then $w\lambda=\lambda$, since
	\begin{equation*}
	\begin{array}{rl}
	\lVert w\lambda\rVert^2&=\lVert (w\lambda-\lambda)+\lambda\rVert^2\\
	&=\lVert w\lambda-\lambda\rVert^2+\lVert\lambda\rVert^2\\
	&\geq \lVert\lambda\rVert^2.
	\end{array}
	\end{equation*}
	But $\lVert w\lambda\rVert^2=\lVert \lambda\rVert^2$, $\lVert w\lambda-\lambda\rVert^2=0$ and thus $w\lambda=\lambda$. 
\end{proof}

\begin{corollary}\label{lambdagroup}
	The set $A(\lambda)$ coincides with $\widetilde{A}(\lambda)/W_\lambda$.
\end{corollary}

As a consequence, if $\{\nu_i\}_i$ is a complete family of representatives for the $W_\lambda$-equivalence classes of the set 

\begin{equation*}
\{\rho-\sum\limits_{\alpha\in A}\alpha\mid A\subseteq \big(\mathbb{R}\{\lambda\}\big)^\perp\cap\Delta^+ \},
\end{equation*}
then the set $\{\lambda+\nu_i\}_i$ is a complete family of representatives for $A(\lambda)$. The following lemma ensures that we can choose $\nu_i$ to be $\Delta^+$-dominant and so it completes the proof of Proposition \ref{basic}.

\begin{lemma}\label{dominant}
	Any element of the form 
	\begin{equation*}\rho-\sum\limits_{\alpha\in A}\alpha\end{equation*}
	with $A\subseteq \big(\mathbb{R}\{\lambda\}\big)^\perp\cap\Delta^+$ is $W_\lambda$-conjugate to a $\Delta^+$-dominant element.
\end{lemma}

\begin{proof}
	Let $\Delta_1^+:=\big(\mathbb{R}\{\lambda\}\big)^\perp\cap\Delta^+$, i.e. the set of the positive roots orthogonal to $\lambda$, and set $\Delta_2^+:=\Delta^+\setminus\Delta_1^+$. The set $\Delta_1:=\Delta_1^+\cup(-\Delta_1^+)$ is a root system. Indeed, if $\alpha$ and $\beta$ are elements of $\Delta_1^+$, then, from the $W$-invariance of the Killing form:
	
	\begin{equation*}
	\begin{array}{rl}
	\langle s_\alpha(\beta),\lambda\rangle&=\langle\beta,s_\alpha(\lambda)\rangle\\
	&=\langle\beta,\lambda\rangle\\
	&=0,
	\end{array}
	\end{equation*}
	and thus $\Delta_1$ is stable with respect to the reflections generated by its elements. The other properties for $\Delta_1$ to be a root system can be easily checked. Let $W_1$ be the Weyl group defined by $\Delta_1$ and $\rho_1$ be the half-sum of the positive root of $\Delta_1$. Then $W_1$ fixes the element $\rho-\rho_1$, i.e. the half-sum of the positive roots included in $\Delta_2:=\Delta_2^+\cup(-\Delta_2^+)$. Note that $\Delta_2$ need not be a root system. Indeed, if $\beta$ is a root of $\Delta_1$ and $\gamma$ is a root in $\Delta_2^+$, then
	
	\begin{equation*}
	\begin{array}{rl}
	\langle s_\beta(\gamma),\lambda\rangle&=\langle\gamma,s_\beta(\lambda)\rangle\\
	&=\langle\gamma,\lambda\rangle\\
	&>0.
	\end{array}
	\end{equation*}
	Thus $W_1\Delta_2^+=\Delta_2^+$, $W_1$ stabilizes $\rho-\rho_1$ and $\langle\rho-\rho_1,\beta\rangle=0$ for any $\beta\in\Delta_1$. Since the Weyl group of a root system acts transitively on its Weyl chambers, if $A\subseteq \Delta_1^+$, then there is some $w_0$ in $W_1$ such that 
	\begin{equation*}
	w_0(\rho_1-\sum\limits_{\alpha\in A}\alpha)
	\end{equation*}
	is $\Delta_1^+$-dominant. In this case, for every $\beta\in\Delta_1^+$,
	
	\begin{equation*}
	\begin{array}{rl}
	\langle w_0(\rho-\sum\limits_{\alpha\in A}\alpha),\beta\rangle&=\langle \rho-\rho_1+w_0(\rho_1-\sum\limits_{\alpha\in A}\alpha),\beta\rangle\\
	&=\langle \rho-\rho_1,\beta\rangle+\langle w_0(\rho_1-\sum\limits_{\alpha\in A}\alpha),\beta\rangle\\
	&=\langle w_0(\rho_1-\sum\limits_{\alpha\in A}\alpha),\beta\rangle\\
	&\geq 0.
	\end{array}
	\end{equation*}
	Therefore $w_0(\rho-\sum\limits_{\alpha\in A}\alpha)$ is $\Delta_1^+$-dominant.

	We will be finished once we show that this is also the case for $\Delta_2^+$. More precisely, we want to show that 
	\begin{equation*}\langle w_0(\rho-\sum\limits_{\alpha\in A}\alpha),\gamma\rangle\geq 0,
	\end{equation*}
	for any $\gamma$ in $\Delta_2^+$. Since $w_0^{-1}\Delta_2^+=\Delta_2^+$, it suffices to show that \begin{equation*}\langle \rho-\sum\limits_{\alpha\in A}\alpha,\gamma\rangle\geq 0,\end{equation*} for any $\gamma$ in $\Delta_2^+$. Suppose that for some $\gamma$ in $\Delta_2^+$, 
	\begin{equation*}
	\langle \rho-\sum\limits_{\alpha\in A}\alpha,\gamma\rangle<0. \end{equation*}
	Then one has that
	\begin{equation}\label{eq1}
	s_\gamma(\rho-\sum\limits_{\alpha\in A}\alpha)=\rho-\sum\limits_{\alpha\in A}\alpha+n\gamma
	\end{equation}
	with $n:=-\frac{2}{\lVert\gamma\rVert^2}\langle \rho-\sum\limits_{\alpha\in A}\alpha,\gamma\rangle$ being positive.

	On the other hand, according to Lemma \ref{otherspin}, the set of weights of $S$ is $W$-stable, and so the element 
	\begin{equation*} s_\gamma(\rho-\sum\limits_{\alpha\in A}\alpha)\end{equation*}
	must be a weight of $S$.  Thus it must be of the form \eqref{Sweight}:
	\begin{equation}\label{eq2}
	s_\gamma(\rho-\sum\limits_{\alpha\in A}\alpha)=\rho-\sum\limits_{\beta\in B}\beta
	\end{equation}
	for some subset $B$ of $\Delta^+$. 
	
	Combining the equalities \eqref{eq1} and \eqref{eq2}, one obtains that
	\begin{equation*}
	\begin{array}{crl}
	&\rho-\sum\limits_{\alpha\in A}\alpha+n\gamma&=\rho-\sum\limits_{\beta\in B}\beta\\
	\Longrightarrow& n\gamma&=\sum\limits_{\alpha\in A}\alpha-\sum\limits_{\beta\in B}\beta\\
	\Longrightarrow&n\langle\gamma,\lambda\rangle&=-\sum\limits_{\beta\in B}\langle\beta,\lambda\rangle.
	\end{array}
	\end{equation*}
	The left-hand side of the last equality is positive, since $\gamma$ has been chosen to be an element of $\Delta_2^+$, while the right-hand side is non positive, since $B$ consists of positive roots and $\lambda$ is dominant. This is a contradiction. Consequently, one has
	\begin{equation*}\langle\rho-\sum\limits_{\alpha\in A}\alpha,\gamma\rangle\geq0\end{equation*} 
	for any $\gamma$ in $\Delta_2^+$ and thus the element 
	$
	w_0(\rho-\sum\limits_{\alpha\in A}\alpha)
	$
	is $\Delta^+$-dominant. 
	
	According to \cite[Proposition 2.72]{Knapp1}, $W_1$ coincides with the stabilizer $W_\lambda$ of $\lambda$ in $W$. Hence the element
	$
	\rho-\sum\limits_{\alpha\in A}\alpha
	$
	turns out to be $W_\lambda$-conjugate with a dominant one.
\end{proof}

This now completes the proof of Proposition \ref{basic}.

\subsection{Noncubic Dirac operators and isotypic components}\label{mainfin}
The cubic Dirac operator $D_{\mathfrak{g},\mathfrak{h}}(V)$ acts in the same way on every irreducible representation of an isotypic component. Namely, on the isotypic component corresponding to a highest weight $\mu\in\mathfrak{t}^*$, $D_{\mathfrak{g},\mathfrak{h}}(V)$ acts by the scalars
\begin{equation*}
\pm\sqrt{\lVert\lambda+\rho\rVert^2-\lVert\mu+\rho_\mathfrak{h}\rVert^2}.
\end{equation*}
As a consequence, either the whole isotypic component belongs to $\ker D_{\mathfrak{g},\mathfrak{h}}(V)$ or its intersection with $\ker D_{\mathfrak{g},\mathfrak{h}}(V)$ is trivial. From Theorem \ref{kerneltcubic}, this is true for every $D^t_{\mathfrak{g},\mathfrak{h}}(V)$, $t\in(0,2)$. On the other hand, according to Theorem \ref{thmbasic}, this turns out to be also true for the noncubic Dirac operator $\widehat{D}_{\mathfrak{g},\mathfrak{h}}(V)$ when $\mathfrak{h}$ coincides with a Cartan subalgebra $\mathfrak{t}$ of $\mathfrak{g}$. 
Note that, since each isotypic component of $V\otimes S$ is generated by "monomials", i.e. by elements of the form $v\otimes u$ with $v$ and $u$ being  weight vectors of $V$ and $S$ respectively, one deduces that this is the case for $\ker D^t_{\mathfrak{g},\mathfrak{h}}(V)$, $t\in(0,2)$, and, in the case when $\mathfrak{h}=\mathfrak{t}$ for $\ker\widehat{D}_{\mathfrak{g},\mathfrak{h}}(V)$.
Nevertheless, as the following example indicates, these properties are not true for $\widehat{D}_{\mathfrak{g},\mathfrak{h}}(V)$ if $\mathfrak{h}\neq\mathfrak{t}$. This complicates the study of $\ker \widehat{D}_{\mathfrak{g},\mathfrak{h}}(V)$ in this case. Consider the following example.

\begin{example}\label{nonpolyn}
	Using the notation of Example \ref{firstexample}, let $\mathfrak{g}:=\mathfrak{sl}(4,\mathbb{C})$ and
	\begin{equation*}\mathfrak{h}:=\mathfrak{t}\oplus\mathfrak{g}_{\varepsilon_3-\varepsilon_4}\oplus\mathfrak{g}_{-\varepsilon_3+\varepsilon_4}\end{equation*}
	with $\mathfrak{q}:=\mathfrak{h}^\perp$ such that
	\begin{equation*}
	\mathfrak{g}=\mathfrak{h}\oplus\mathfrak{q}.
	\end{equation*}
	Consider the Clifford algebra $\mathbf{C}(\mathfrak{q})$ of $\mathfrak{q}$ and the space of spinors $S$ for $\mathbf{C}(\mathfrak{q})$, defined by the negative root spaces of $\mathfrak{q}$. For every $\alpha\in\Delta^+$, choose
	$e_{\pm\alpha}\in\mathfrak{g}_{\pm\alpha}$ such that $\langle e_\alpha,e_{-\alpha}\rangle=1$.
	Furthermore, let $V:=\mathbb{C}^4$ be the standard representation of $\mathfrak{g}$, i.e. the irreducible $\mathfrak{g}$-representation of highest weight 
	\begin{equation*}
	\lambda=\frac{3}{4}\varepsilon_1-\frac{1}{4}\varepsilon_2-\frac{1}{4}\varepsilon_3-\frac{1}{4}\varepsilon_4.
	\end{equation*}
	Then, the noncubic Dirac operator $\widehat{D}_{\mathfrak{g},\mathfrak{h}}(V)$ is given by
	\begin{align*}
	\widehat{D}_{\mathfrak{g},\mathfrak{h}}(V)=\sqrt{2}\{&e_{\varepsilon_1-\varepsilon_2}\otimes \gamma(e_{-\varepsilon_1+\varepsilon_2})
	+e_{\varepsilon_1-\varepsilon_3}\otimes \gamma(e_{-\varepsilon_1+\varepsilon_3})\\
	+&e_{\varepsilon_1-\varepsilon_4}\otimes \gamma(e_{-\varepsilon_1+\varepsilon_4})
	+e_{\varepsilon_2-\varepsilon_3}\otimes \gamma(e_{-\varepsilon_2+\varepsilon_3})\\
	+&e_{\varepsilon_2-\varepsilon_4}\otimes \gamma(e_{-\varepsilon_2+\varepsilon_4})
	+e_{-\varepsilon_1+\varepsilon_2}\otimes \gamma(e_{\varepsilon_1-\varepsilon_2})\\
	+&e_{-\varepsilon_1+\varepsilon_3}\otimes \gamma(e_{\varepsilon_1-\varepsilon_3})
	+e_{-\varepsilon_1+\varepsilon_4}\otimes \gamma(e_{\varepsilon_1-\varepsilon_4})\\
	+&e_{-\varepsilon_2+\varepsilon_3}\otimes \gamma(e_{\varepsilon_2-\varepsilon_3})
	+e_{-\varepsilon_2+\varepsilon_4}\otimes \gamma(e_{\varepsilon_2-\varepsilon_4})\}.
	\end{align*}
	If $\{v_1,v_2,v_3,v_4\}$ is the standard basis of $\mathbb{C}^4$, set 
	\begin{align*}
	x_1&:=v_3\otimes e_{-\varepsilon_1+\varepsilon_3}\wedge e_{-\varepsilon_1+\varepsilon_4}\wedge e_{-\varepsilon_2+\varepsilon_4}\\x_2&:=v_4\otimes e_{-\varepsilon_1+\varepsilon_3}\wedge e_{-\varepsilon_1+\varepsilon_4}\wedge e_{-\varepsilon_2+\varepsilon_3}.
	\end{align*}
	Then
	\begin{align*}
	\widehat{D}_{\mathfrak{g},\mathfrak{h}}(V)x_1&=v_2\otimes e_{-\varepsilon_1+\varepsilon_3}\wedge e_{-\varepsilon_1+\varepsilon_4}\wedge e_{-\varepsilon_2+\varepsilon_3}\wedge e_{-\varepsilon_2+\varepsilon_4}\neq0\\
	\widehat{D}_{\mathfrak{g},\mathfrak{h}}(V)x_2&=-v_2\otimes e_{-\varepsilon_1+\varepsilon_3}\wedge e_{-\varepsilon_1+\varepsilon_4}\wedge e_{-\varepsilon_2+\varepsilon_3}\wedge e_{-\varepsilon_2+\varepsilon_4}\neq0
	\end{align*}
	but
	\begin{equation*}
	\widehat{D}_{\mathfrak{g},\mathfrak{h}}(V)(x_1+x_2)=0.
	\end{equation*}
	Therefore, the elements $x_1$, $x_2$ and $x_1+x_2$ belong to the same isotypic component of $V\otimes S$, the element $x_1+x_2$ belonging to $\ker\widehat{D}_{\mathfrak{g},\mathfrak{h}}(V)$, while $x_1$ and $x_2$ do not. In other words, only a part of this isotypic component belongs to $\ker\widehat{D}_{\mathfrak{g},\mathfrak{h}}(V)$.
\end{example}

\section{Noncubic Dirac operators for classical Lie algebras}\label{sectionclassical}
As before, suppose that $\mathfrak{g}$ is complex and semisimple while $\mathfrak{h}$ is a Cartan subalgebra $\mathfrak{t}$ of $\mathfrak{g}$. Let $V$ be an irreducible finite-dimensional representation of $\mathfrak{g}$ and $\Pi(V)$ the set of all weights of $V$.
For every weight $\nu\in\Pi(V)$, set 
\begin{equation*}
\begin{array}{rrl}
&^1\Delta^{(\nu)}&:=\{\alpha\in\Delta^+\mid \nu+\alpha\in\Pi(V)\},\\
&^2\Delta^{(\nu)}&:=\{\alpha\in\Delta^+\mid \nu-\alpha\in\Pi(V)\},\\
\text{and}& \mathcal{P}^{(\nu)}&:=\{I\in \mathcal{P}(\Delta^+)\mid
{}^1\Delta^{(\nu)}\subseteq I \text{ and } {}^2\Delta^{(\nu)}\cap I=\emptyset \},
\end{array}
\end{equation*}
where $\mathcal{P}(\Delta^+)$ stands for the power set of $\Delta^+$.

\begin{definition}\index{property $(*)$}
	We say that a representation $V$ of $\mathfrak{g}$ satisfies property $(*)$, if for every weight $\nu$ of $V$ and every root $\alpha$ of ${}^1\hspace{-0.5mm}\Delta^{(\nu)}$ (respectively of ${}^2\hspace{-0.5mm}\Delta^{(\nu)}$), we have $\pi(e_\alpha)v_\nu\neq 0$ (respectively $\pi(e_{-\alpha})v_\nu\neq 0$) for every nonzero weight vector $v_\nu$ of $V_\nu$. In other words, for $\nu$ and $\nu+\alpha$ (respectively $\nu-\alpha$) being in $\Pi(V)$, the linear map
	\begin{equation*}
	\begin{array}{rl}
	\pi(e_\alpha)&:V_\nu\longrightarrow V_{\nu+\alpha}\\
	(\text{respectively } \pi(e_{-\alpha)}&:V_\nu\longrightarrow V_{\nu-\alpha})
	\end{array}
	\end{equation*}
	is injective.
\end{definition}

\begin{theorem}\label{property*}
	Suppose $\nu$, ${}^1\Delta^{(\nu)}$ and ${}^2\Delta^{(\nu)}$ are as above and $V$ satisfies property $(*)$. Then, one has
	\begin{equation*}
	\ker\widehat{D}_{\mathfrak{g},\mathfrak{t}}(V)=\bigoplus\limits_{\nu\in\Pi(V)}\langle v_\nu\otimes e_I\mid I\in \mathcal{P}^{(\nu)}\rangle.
	\end{equation*} 
\end{theorem}

\begin{proof}
	Let $v_\nu\otimes e_I$ be a nonzero weight vector of $V\otimes S$ and $\alpha$ a positive root of $\Delta^+$. If $\nu+\alpha$ is not a weight of $V$, i.e. $\alpha\notin{}^1\Delta^{(\nu)}$, then $\pi(e_a)v_\nu=0$ and so $\pi(e_\alpha)\otimes \gamma(e_{-\alpha})(v_\nu\otimes e_I)=0$. On the other hand, if $\nu+\alpha$ is a weight of $V$, i.e. $\alpha\in{}^1\Delta^{(\nu)}$, and since $V$ satisfies property $(*)$, $\pi(e_\alpha)v_\nu\neq 0$. So $\pi(e_\alpha)\otimes \gamma(e_{-\alpha})(v_\nu\otimes e_I)=0$ if and only if $\gamma(e_{-\alpha})e_I=0$ or equivalently if and only if $I$ contains $\alpha$. Summarizing, we have
	
	\begin{equation}
	\pi(e_\alpha)\otimes \gamma(e_{-\alpha})(v_\nu\otimes e_I)
	\begin{cases}
	=0 \quad \text{if } \alpha\notin{}^{1}\Delta^{(\nu)}\text{ or }\alpha\in I\\
	\neq 0\quad \text{if } \alpha\in{}^{1}\Delta^{(\nu)}\text{ and }\alpha\notin I\\
	\end{cases}
	\end{equation}
	Using a similar argument, we obtain
	\begin{equation}
	\pi(e_{-\alpha})\otimes \gamma(e_{\alpha})(v_\nu\otimes e_I)
	\begin{cases}
	=0 \quad \text{if } \alpha\notin{}^{2}\Delta^{(\nu)}\text{ or }\alpha\notin I\\
	\neq 0\quad \text{if } \alpha\in{}^{2}\Delta^{(\nu)}\text{ and }\alpha\in I\\
	\end{cases}
	\end{equation}
	Therefore, the element $v_\nu\otimes e_I$ belongs to \begin{equation*}\bigcap\limits_{\alpha\in\Delta}\ker(\pi(e_\alpha)\otimes \gamma(e_{-\alpha}))\end{equation*}
	if and only if for every $\alpha$ in ${}^{1}\Delta^{(\nu)}$ we have $\alpha$ belonging to $I$ and for every $\alpha$ in ${}^{2}\Delta^{(\nu)}$ we have $\alpha$ not belonging to $I$. This is equivalent to say that ${}^1\Delta^{(\nu)}\subseteq I$ and ${}^2\Delta^{(\nu)}\cap I=\emptyset$, i.e. $I\in\mathcal{P}^{(\nu)}$.
	Hence, 
	
	\begin{equation*}
	\bigcap\limits_{\alpha\in\Delta}\ker(\pi(e_\alpha)\otimes \gamma(e_{-\alpha}))=\bigoplus\limits_{\nu\in \Pi(V)}\langle v_\nu\otimes e_I\mid I\in \mathcal{P}^{(\nu)}\rangle.
	\end{equation*}
	Using Proposition \ref{intersectionD}, we obtain that
	
	\begin{equation*}
	\ker\widehat{D}_{\mathfrak{g},\mathfrak{t}}(V)=\bigoplus\limits_{\nu\in \Pi(V)}\langle v_\nu\otimes e_I\mid I\in \mathcal{P}^{(\nu)}\rangle.\qedhere
	\end{equation*}
\end{proof}

Let us illustrate the situation with an example. 

\begin{example}
Let $\mathfrak{g}$ be
$\mathfrak{sl}(n,\mathbb{C})$, i.e. the Lie algebra of all $n-$by$-n$ complex traceless matrices.
The root system of $\mathfrak{g}$ is \begin{equation*}\Delta=\{\pm(\varepsilon_i-\varepsilon_j)\mid 1\leq i<j\leq n  \}.
\end{equation*} 
We choose \begin{equation*}\Delta^+=\{\varepsilon_i-\varepsilon_j\mid 1\leq i<j\leq n \}
\end{equation*}
to be the set of positive roots of $\Delta$
and let $V$ be the standard representation of $\mathfrak{g}$.
Then $V$ is irreducible, of highest weight
\begin{equation*}
\lambda=\frac{n}{n+1}\varepsilon_1-\frac{1}{n+1}\varepsilon_2-\ldots-\frac{1}{n+1}\varepsilon_{n+1}
\end{equation*} 
while
\begin{equation*}
\Pi(V)=\{\mu_i:=\lambda-\varepsilon_1+\varepsilon_i\mid 1\leq i\leq n\}.
\end{equation*}
One can check that $V$ satisfies property $(*)$ and
\begin{equation*}
\begin{array}{rrl}
&^1\Delta^{(\mu_k)}&=\{\varepsilon_i-\varepsilon_k\mid 1\leq i<k\},\\
\text{and}&^2\Delta^{(\mu_k)}&=\{\varepsilon_k-\varepsilon_l\mid k<l\leq n+1\}.
\end{array}
\end{equation*}
Therefore, if $v_1,\ldots,v_n$ are the weight vectors of $\mu_1,\ldots,\mu_n$ respectively, according to Theorem \ref{property*}:
\begin{equation*}
\ker \widehat{D}_{\mathfrak{g},\mathfrak{t}}(V)=\bigoplus\limits_{1\leq k\leq n}\langle v_k\otimes e_{-\varepsilon_1+\varepsilon_k}\wedge\ldots\wedge e_{-\varepsilon_{k-1}+\varepsilon_{k}}\wedge e_J\mid J\subseteq \Delta^+\setminus\big({}^1\Delta^{(\mu_k)}\cup {}^2\Delta^{(\mu_k)}\big)\rangle.
\end{equation*}
In particular, one deduces that 
\begin{equation*}\dim\ker\widehat{D}_{\mathfrak{g},\mathfrak{t}}(V)=n\times 2^{\frac{(n-1)(n-2)}{2}},
\end{equation*}
while, according to \eqref{kernel}, the dimension of the cubic Dirac operator $D_{\mathfrak{g},\mathfrak{h}}(V)$ is:
\begin{equation*}
\dim\ker D_{\mathfrak{g},\mathfrak{t}}(V)=n!
\end{equation*}
\end{example}

In a similar way, one can proceed with the other classical Lie algebras.
The following table contains the dimensions of the kernels $\ker D_{\mathfrak{g},\mathfrak{t}}(V)$ and $\ker\widehat{D}_{\mathfrak{g},\mathfrak{t}}(V)$ in these cases.

{ \renewcommand*{\arraystretch}{1.9}
	\begin{table}[H]
		\centering
		\begin{tabular}{|c|c|c|}
			\hline
			$\mathfrak{g}$ &$\dim\ker D_{\mathfrak{g},\mathfrak{t}}(V)=\lvert W\rvert$ &$\dim\ker\widehat{D}_{\mathfrak{g},\mathfrak{t}}(V)$ \\
			\hhline{|=|=|=|}
			$\mathfrak{sl}(n,\mathbb{C})$&$n!$&
			$	n\times 2^{\frac{(n-1)(n-2)}{2}}$\\
			$\mathfrak{so}(2n+1,\mathbb{C})$&$2^nn!$&$2n\times 2^{(n-1)^2}$\\
			$\mathfrak{sp}(n,\mathbb{C})$&$2^nn!$&$2n\times 2^{(n-1)^2}$\\
			$\mathfrak{so}(2n,\mathbb{C})$&$2^{n-1}n!$&$2n\times 2^{(n-1)(n-2)}$\\
			\hline
		\end{tabular}
		\captionsetup{justification=centering}
		\caption{Dimensions of the kernels of cubic and noncubic Dirac operators for the standard representation}\label{dimkernels}
	\end{table}
}

\section{Noncubic Dirac operators for exceptional Lie algebras}\label{sectionexceptional}
In this section, we discuss the kernel of noncubic Dirac operators for exceptional Lie algebras.
Let $\mathfrak{g}$ be an exceptional Lie algebra and choose $\mathfrak{h}$ to be a Cartan subalgebra $\mathfrak{t}$ of $\mathfrak{g}$. Let $\Delta$ be the root system of $\mathfrak{g}$, $\Pi:=\{\alpha_1,\ldots,\alpha_l \}$ a set of simple roots of $\Delta$ and $\Delta^+$ the set of positive roots determined by $\Pi$. 
We wish to describe the kernel of the noncubic Dirac operator $\widehat{D}_{\mathfrak{g},\mathfrak{t}}(V)$ when $V$ is the standard representation of $\mathfrak{g}$. This is the case when the highest weight $\lambda$ of the $\mathfrak{g}$-representation $V$ is the first fundamental weight $\bar{w}_1$ of $\mathfrak{g}$. As in the proof of Lemma \ref{dominant}, set
\begin{equation*}
\begin{array}{rl}
\Delta_1&:=\{\alpha\in\Delta\mid \langle\lambda,\alpha\rangle=0 \},\\
\Delta_2&:=\Delta\setminus\Delta_1,\\
\Delta_i^+&:=\Delta_i\cap\Delta^+,\quad i=1,2.
\end{array}
\end{equation*}
Let $\rho_1$ be the half-sum of the positive roots of $\Delta_1^+$. As we have already shown in the proof of Lemma \ref{dominant}, $\Delta_1$ is a root system with $\Delta_1^+$ being the set of positive roots. Moreover, one can verify that $\Delta_1^+=\Delta^+\cap\langle\alpha_2,\ldots,\alpha_l\rangle$ and thus $\Pi_1:=\{\alpha_2,\ldots,\alpha_l \}$ is the set of simple roots for $\Delta_1^+$. Consequently, the Dynkin diagram of $\Delta_1$ is obtained from the Dynkin diagram of $\Delta$ by deleting the edge corresponding to the simple root $\alpha_1$.

To calculate $\ker\widehat{D}_{\mathfrak{g},\mathfrak{t}}(V)$, according to Theorem \ref{thmbasic}, one has to determine the set $A(\lambda)$. In other words, using Proposition \ref{basic}, it suffices to find all $\Delta^+$-dominant weights of the form
\begin{equation*}
\rho-\sum\limits_{\alpha\in A}\alpha,\quad A\subseteq\Delta_1^+.
\end{equation*}
To find these weights, we will need the following lemma.

\begin{lemma}\label{lemdominant}
	With the above notation, let $A$ be a subset of $\Delta_1^+$. The weight $\rho-\sum\limits_{\alpha\in A}\alpha$ is $\Delta^+$-dominant if and only if the weight $\rho_1-\sum\limits_{\alpha\in A}\alpha$ is $\Delta^+_1$-dominant.
\end{lemma}

\begin{proof}
	If $\rho-\sum\limits_{\alpha\in A}\alpha$ is $\Delta^+$-dominant, obviously it is $\Delta_1^+$-dominant. Moreover, $\rho_1-\sum\limits_{\alpha\in A}\alpha$ is $\Delta_1^+$-dominant. Indeed, as we have already seen in the proof of Proposition \ref{basic}, $\langle\rho-\rho_1,\beta\rangle=0$ for every $\beta$ in $\Delta_1^+$ and so 
	
	\begin{equation*}\label{equat}
	\begin{array}{rl}
	\langle\rho_1-\sum\limits_{\alpha\in A}\alpha,\beta\rangle&=\langle(\rho-\rho_1)+\rho-\sum\limits_{\alpha\in A}\alpha,\beta\rangle\\
	&=\langle\rho-\sum\limits_{\alpha\in A}\alpha,\beta\rangle.
	\end{array}
	\end{equation*}
	Therefore, $\rho_1-\sum\limits_{\alpha\in A}\alpha$ is $\Delta^+_1$-dominant if and only if $\rho-\sum\limits_{\alpha\in A}\alpha$ is $\Delta_1^+$-dominant, which is the case for $\rho-\sum\limits_{\alpha\in A}\alpha$ being $\Delta^+$-dominant.
	
	On the other hand, if $\rho_1-\sum\limits_{\alpha\in A}\alpha$ is $\Delta_1^+$-dominant, then, as we have seen in the proof of Proposition \ref{basic}, $\rho-\sum\limits_{\alpha\in A}\alpha$ is $\Delta^+$-dominant. 
\end{proof}

Using the above lemma, in order to determine $A(\lambda)$, it suffices to find all the $\Delta_1^+$-dominant weights of the form
\begin{equation}\label{domform}
\rho_1-\sum\limits_{\alpha\in A}\alpha, \quad A\subseteq \Delta_1^+.
\end{equation}
%
Let $\mathfrak{g}_1$ be the Lie algebra corresponding to $\Delta_1$ and choose $\mathfrak{h}_1$ to be a Cartan subalgebra of $\mathfrak{g}_1$. Then $\mathfrak{q}_1:=\mathfrak{h}_1^\perp=\mathfrak{n}_1\oplus \mathfrak{n}_1^-$, where $\mathfrak{n}_1$ (respectively $\mathfrak{n}_1^-$) is the sum of all positive (respectively negative) root subspaces of $\mathfrak{g}_1$. 
Then the $\Delta_1^+$-dominant weights of the form \eqref{domform} are exactly the $\Delta_1^+$-dominant weights of the space of spinors $S_1$ of the Clifford algebra $\mathbf{C}(\mathfrak{n}_1\oplus\mathfrak{n}_1^-)$.

We illustrate the above discussion with the following example. Let $\mathfrak{g}$ be the exceptional Lie algebra $F_4$ and $V$ the standard representation of $\mathfrak{g}$. The root system of $\mathfrak{g}$ is
\begin{equation*}
\Delta:=\{\pm\varepsilon_i\pm\varepsilon_j\mid 1\leq i<j\leq 4 \}\cup\{\pm\varepsilon_i\mid 1\leq i\leq 4 \}\cup\big\{\frac{1}{2}(\pm\varepsilon_1\pm\varepsilon_2\pm\varepsilon_3\pm\varepsilon_4) \big\},
\end{equation*}
where $\varepsilon_i$ is the element of the dual space of $\mathbb{R}^4$ sending the $i$-th element of the standard basis to $1$ and all the others to $0$. The set $\Delta$ is the root system of $F_4$ and we fix 
\begin{equation*}
\Delta^+:=\{\varepsilon_i\pm\varepsilon_j\mid 1\leq i<j\leq 4 \}\cup\{\varepsilon_i\mid 1\leq i\leq 4 \}\cup\{\frac{1}{2}(\varepsilon_1\pm\varepsilon_2\pm\varepsilon_3\pm\varepsilon_4) \}
\end{equation*}
to be the set of positive roots of $\Delta$. Then, \begin{equation*}\Pi:=\{\alpha_1:=\frac{1}{2}(\varepsilon_1-\varepsilon_2-\varepsilon_3-\varepsilon_4),\alpha_2:=\varepsilon_4, \alpha_3:=\varepsilon_3-\varepsilon_4,\alpha_4:=\varepsilon_2-\varepsilon_3 \}
\end{equation*}
is the set of simple roots of $\Delta$. The highest weight $\lambda$ of the standard representation is $\lambda=\bar{\omega}_1=\varepsilon_1$ and the set of roots being orthogonal with $\lambda$ is 
\begin{equation*}
\Delta_1=\{\pm\varepsilon_i\pm\varepsilon_j\mid 2\leq i<j\leq 4 \}\cup\{\pm\varepsilon_i\mid 2\leq i\leq 4 \}.
\end{equation*}
Then, the half sum of the positive roots of $\Delta_1^+$ is $\rho_1=\frac{5}{2}\varepsilon_2+\frac{3}{2}\varepsilon_3+\frac{1}{2}\varepsilon_4$ and 
\begin{equation*}
\Pi_1:=\{\alpha_2,\alpha_3,\alpha_4 \}
\end{equation*}
is the set of simple roots of $\Delta_1$. The Dynkin diagram of $F_4$ is 

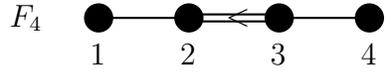
\begin{equation*}
\text{
	\begin{tikzpicture}[scale=.6]
	\draw (-3,0) node[anchor=east]  {$F_4$};
	\draw (-1.6,-0.8) node[anchor=east]  {$1$};
	\draw (0.4,-0.8) node[anchor=east]  {$2$};
	\draw (2.4,-0.8) node[anchor=east]  {$3$};
	\draw (4.4,-0.8) node[anchor=east]  {$4$};
	\draw (1.6,0) node[anchor=east]  {$<$};
	\draw[thick,fill=black] (-2 cm ,0) circle (.3 cm);
	\draw[thick,fill=black] (0 ,0) circle (.3 cm);
	\draw[thick,fill=black] (2 cm,0) circle (.3 cm);
	\draw[thick,fill=black] (4 cm,0) circle (.3 cm);
	\draw[thick] (15: 3mm) -- +(1.5 cm, 0);
	\draw[xshift=-2 cm,thick] (0: 3 mm) -- +(1.4 cm, 0);
	\draw[thick] (-15: 3 mm) -- +(1.5 cm, 0);
	\draw[xshift=2 cm,thick] (0: 3 mm) -- +(1.4 cm, 0);
	\end{tikzpicture}}.
\end{equation*}
By deleting the edge corresponding to $\alpha_1$, we obtain the
Dynkin diagram of $\Delta_1$:

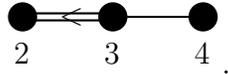
\begin{equation*}
\text{
	\begin{tikzpicture}[scale=.6] 
	\draw (0.4,-0.8) node[anchor=east]  {$2$};
	\draw (2.4,-0.8) node[anchor=east]  {$3$};
	\draw (4.4,-0.8) node[anchor=east]  {$4$};
	\draw (1.6,0) node[anchor=east]  {$<$};
	\draw[thick,fill=black] (0 ,0) circle (.3 cm);
	\draw[thick,fill=black] (2 cm,0) circle (.3 cm);
	\draw[thick,fill=black] (4 cm,0) circle (.3 cm);
	\draw[thick] (15: 3mm) -- +(1.5 cm, 0);
	\draw[thick] (-15: 3 mm) -- +(1.5 cm, 0);
	\draw[xshift=2 cm,thick] (0: 3 mm) -- +(1.4 cm, 0);
	\end{tikzpicture}}.
\end{equation*}
This is exactly the Dynkin diagram of the Lie algebra $B_3$. With the previous notation, the $\Delta_1^+$-dominant weights of the spin module $S_1$ for $B_3$ are given in Table \ref{F4table}.

{ \renewcommand*{\arraystretch}{1.9}
	\begin{table}[H]
		\centering
		\begin{tabular}{|c|c|}
			\hline
			$\text{Weight}$ & \text{Vectors}\\
			\hhline{|=|=|}
			$\frac{5}{2}\varepsilon_2+\frac{3}{2}\varepsilon_3+\frac{1}{2}\varepsilon_4 $ & $1$\\
			\hline
			$\frac{5}{2}\varepsilon_2+\frac{1}{2}\varepsilon_3+\frac{1}{2}\varepsilon_4 $ & $\substack{\quad\\e_{-\varepsilon_3}\\\quad\\e_{-\varepsilon_3+\varepsilon_4}\wedge e_{-\varepsilon_4}\\
				\quad}$ \\
			\hline
			$\frac{3}{2}\varepsilon_2+\frac{3}{2}\varepsilon_3+\frac{3}{2}\varepsilon_4 $& $\substack{\quad\\e_{-\varepsilon_2+\varepsilon_3}\wedge e_{-\varepsilon_3+\varepsilon_4}\\\quad \\e_{-\varepsilon_2+\varepsilon_4}\\\quad }$\\
			\hline
			$\frac{3}{2}\varepsilon_2+\frac{3}{2}\varepsilon_3+\frac{1}{2}\varepsilon_4 $& $\substack{\quad\\e_{-\varepsilon_2}\\\quad\\e_{-\varepsilon_2+\varepsilon_3}\wedge e_{-\varepsilon_3}\\\quad\\e_{-\varepsilon_2+\varepsilon_4}\wedge e_{-\varepsilon_4}\\\quad\\e_{-\varepsilon_2+\varepsilon_3}\wedge e_{-\varepsilon_3+\varepsilon_4}\wedge e_{-\varepsilon_4} \\
				\quad}$\\
			\hline
			$\frac{3}{2}\varepsilon_2+\frac{1}{2}\varepsilon_3+\frac{1}{2}\varepsilon_4 $& $\substack{\quad\\e_{-\varepsilon_2-\varepsilon_3} \\
				\quad\\e_{-\varepsilon_2}\wedge e_{-\varepsilon_3}\\
				\quad\\e_{-\varepsilon_3}\wedge e_{-\varepsilon_4}\wedge e_{-\varepsilon_2+\varepsilon_4} \\
				\quad\\e_{-\varepsilon_2}\wedge e_{-\varepsilon_3+\varepsilon_4}\wedge e_{-\varepsilon_4}\\\quad\\e_{-\varepsilon_3+\varepsilon_4}\wedge e_{-\varepsilon_2-\varepsilon_4} \\
				\quad\\e_{-\varepsilon_2-\varepsilon_3}\wedge e_{-\varepsilon_3+\varepsilon_4}\wedge e_{-\varepsilon_3-\varepsilon_4} \\
				\quad\\e_{-\varepsilon_2+\varepsilon_4}\wedge e_{-\varepsilon_3-\varepsilon_4}\\\quad\\e_{-\varepsilon_3}\wedge e_{-\varepsilon_4}\wedge e_{-\varepsilon_2+\varepsilon_3}\wedge e_{-\varepsilon_3+\varepsilon_4}\\\quad}$\\
			\hline
			$\frac{1}{2}\varepsilon_2+\frac{1}{2}\varepsilon_3+\frac{1}{2}\varepsilon_4 $& $\substack{\quad\\e_{-\varepsilon_2}\wedge e_{-\varepsilon_2-\varepsilon_3} \\
				\quad\\e_{-\varepsilon_2}\wedge e_{-\varepsilon_3+\varepsilon_4}\wedge e_{-\varepsilon_2-\varepsilon_4}\\
				\quad\\e_{-\varepsilon_2}\wedge e_{-\varepsilon_2+\varepsilon_4}\wedge e_{-\varepsilon_3-\varepsilon_4} \\
				\quad\\e_{-\varepsilon_4}\wedge e_{-\varepsilon_2+\varepsilon_4}\wedge e_{-\varepsilon_2-\varepsilon_3}\\
				\quad\\e_{-\varepsilon_3}\wedge e_{-\varepsilon_2+\varepsilon_4}\wedge e_{-\varepsilon_2-\varepsilon_4} \\
				\quad\\e_{-\varepsilon_3}\wedge e_{-\varepsilon_2+\varepsilon_3}\wedge e_{-\varepsilon_2-\varepsilon_3} \\
				\quad\\e_{-\varepsilon_2}\wedge e_{-\varepsilon_3}\wedge e_{-\varepsilon_4}\wedge e_{-\varepsilon_2+\varepsilon_4}\\
				\quad\\e_{-\varepsilon_3}\wedge e_{-\varepsilon_2+\varepsilon_3}\wedge e_{-\varepsilon_2+\varepsilon_4}\wedge e_{-\varepsilon_3-\varepsilon_4}\\
				\quad\\e_{-\varepsilon_3}\wedge e_{-\varepsilon_2+\varepsilon_3}\wedge e_{-\varepsilon_3+\varepsilon_4}\wedge e_{-\varepsilon_2-\varepsilon_4}\\
				\quad\\e_{-\varepsilon_2}\wedge e_{-\varepsilon_2+\varepsilon_3}\wedge e_{-\varepsilon_3+\varepsilon_4}\wedge e_{-\varepsilon_3-\varepsilon_4}\\
				\quad\\e_{-\varepsilon_4}\wedge e_{-\varepsilon_2+\varepsilon_4}\wedge e_{-\varepsilon_3+\varepsilon_4}\wedge e_{-\varepsilon_2-\varepsilon_4}\\
				\quad\\e_{-\varepsilon_4}\wedge e_{-\varepsilon_2+\varepsilon_3}\wedge e_{-\varepsilon_3+\varepsilon_4}\wedge e_{-\varepsilon_2-\varepsilon_3}\\
				\quad\\e_{-\varepsilon_2}\wedge e_{-\varepsilon_3}\wedge e_{-\varepsilon_4}\wedge e_{-\varepsilon_2+\varepsilon_3}\wedge e_{-\varepsilon_3+\varepsilon_4}\\
				\quad\\e_{-\varepsilon_4}\wedge e_{-\varepsilon_2+\varepsilon_3}\wedge e_{-\varepsilon_2+\varepsilon_4}\wedge e_{-\varepsilon_3+\varepsilon_4}\wedge e_{-\varepsilon_3-\varepsilon_4}\\
				
				\quad
			}$\\
			\hline
			\end{tabular}
			\caption{$\Delta_1^+$-dominant weights of $S_1$}\label{F4table}
			\end{table}
		}
		
		According to the discussion at the beginning of this chapter, the set $A(\lambda)$ consists of all elements of the form $\rho-\rho_1+\mu_1$, with $\mu_1$ being a weight of Table \ref{F4table}. Using Theorem \ref{thmbasic}, the kernel of the noncubic Dirac operator $\widehat{D}_{\mathfrak{g},\mathfrak{h}}(V)$ is 
		
		\begin{equation*}
		\ker\widehat{D}_{\mathfrak{g},\mathfrak{t}}(V)=\bigoplus\limits_{\mu_1}(V\otimes S)^{[(\rho-\rho_1)+\mu_1]},
		\end{equation*} 
		where $\mu_1$ runs over the set of the weights of Table \ref{F4table}. The dimension of each isotypic component $(V\otimes S)^{[(\rho-\rho_1)+\mu_1]}$ can be deduced by this table.
		
		In \cite{thesis}, we prove similar results for the other exceptional Lie algebras.
		
		\section{Application to Slebarski's Dirac operators}\label{sectionSleb}
		In \cite{slebarski}, Slebarski studied a family of Dirac operators acting on functions defined over compact homogeneous spaces arising by a family of connections. A careful computation and discussion of their squares was provided by Agricola in the more general setting of naturally reductive homogeneous spaces \cite[Theorem 3.2]{Ilka}. She also established precise links between these operators and string theory in physics \cite[Section 4]{Ilka}. In what follows, we study representation-theoretic analogues of these operators.
		
		Let $G$ be a real compact connected semisimple Lie group and $H$ a closed subgroup of $G$ with complexified Lie algebras $\mathfrak{g}$ and $\mathfrak{h}$ respectively. Suppose that there is a common Cartan subalgebra $\mathfrak{t}$ of both $\mathfrak{g}$ and $\mathfrak{h}$. Let $\Delta$ and $\Delta_\mathfrak{h}$ be the corresponding root systems and $\Delta^+$ and $\Delta_\mathfrak{h}^+$ positive systems for $\Delta$ and $\Delta_\mathfrak{h}$ respectively such that $\Delta_\mathfrak{h}^+\subset\Delta^+$. Moreover, suppose that $\mathfrak{h}$ satisfies the condition \eqref{conditionh}
		and let 
		\begin{equation*}
		\mathfrak{g}=\mathfrak{h}\oplus\mathfrak{q} \text{ with }\mathfrak{q}=\mathfrak{h}^\perp.
		\end{equation*}
		Consider the Clifford algebra $\mathbf{C}(\mathfrak{q})$ of $\mathfrak{q}$ and let $S$ be a space of spinors for $\mathbf{C}(\mathfrak{q})$. Then $\mathfrak{h}$ acts on $S$ by the spin representation \eqref{haction}. This action integrates to an action of the spin double cover $\widetilde{H}$ of $H$. Let $E$ be a finite-dimensional irreducible representation of $\widetilde{H}$ such that $S\otimes E$ is a representation of $H$. Let $\mathcal{C}^\infty(G)$ be the space of complex-valued smooth functions defined on $G$. Then $G$ acts smoothly on $\mathcal{C}^\infty(G)$ by left and right translations and this action can be extended on the space $L^2(G)$ of square-integrable functions of $G$. Let 
		$	\left[L^2(G)\otimes (S\otimes E)\right]^H$
		be the space of $H$-invariants with the $H$-action on $L^2(G)$ given by right translations. Then $G$ acts on $\left[L^2(G)\otimes (S\otimes E)\right]^H$ by left-translations on $L^2(G)$.
		
		\begin{definition}[Geometric cubic and $t$-noncubic Dirac operators]\label{geomcub}\index{Dirac operator!geometric cubic}
			For $t\in[0,2]$, let $\mathcal{D}^t_{G/H}(E)$ be the operator \begin{equation*}\mathcal{D}^t_{G/H}(E):\left[L^2(G)\otimes (S\otimes E)\right]^H\rightarrow \left[L^2(G)\otimes (S\otimes E)\right]^H
			\end{equation*} given by
			\begin{equation*}
			\mathcal{D}^t_{G/H}(E)=\sqrt{2}\big\{\sum\limits_{i,j}\langle\tilde{e}_i,\tilde{e}_j\rangle r(e_i)\otimes \gamma(e_j)\otimes 1-t\cdot1\otimes \gamma(c)\otimes 1\big\},
			\end{equation*}
			where
			$\{\tilde{e}_i\}$ and $\{e_i\}$ are dual bases of $\mathfrak{q}$, $r(\cdot)$ the differential of right translations and $\gamma(c)$ the cubic term \eqref{cubictermm}. As in Definition \ref{noncubic}, for $t\in[0,1)\cup(1,2]$, we will call $\mathcal{D}^t_{G/H}(E)$ the geometric $t$-noncubic Dirac operator, while $\mathcal{D}_{G/H}(E):=\mathcal{D}^1_{G/H}(E)$ is called the geometric cubic Dirac operator. For $t=0$, the operator $\widehat{\mathcal{D}}_{G/H}(E):=\mathcal{D}^0_{G/H}(E)$ is called the geometric noncubic Dirac operator.		\end{definition}
		The operators $\mathcal{D}^t_{G/H}(E)$, $t\in[0,2]$, do not depend on the choice of the bases $\{e_i\}$ and $\{\tilde{e}_i\}$ and all are $G$-equivariant so that their kernels are representations of $G$. Moreover, in the case of $\mathcal{D}_{G/H}(E)$, in analogy to \eqref{Dsquare}, one has
		\begin{equation*}
		\mathcal{D}_{G/H}(E)^2=r(\Omega_\mathfrak{g})\otimes1\otimes1-\big(r\otimes \gamma\big)(\Omega_{\mathfrak{h}_\Delta})\otimes1+\big(\lVert\rho\rVert^2-\lVert\rho_\mathfrak{h}\rVert^2\big).
		\end{equation*}
		
		There is a connection between algebraic and geometric Dirac operators. More precisely, using Peter-Weyl theorem, one obtains a $G$-equivariant isomorphism \cite{thesis}
		\begin{equation}\label{equivpw}
		\Phi:\left[L^2(G)\otimes E\right]^H\xrightarrow{\hspace{1mm}\sim\hspace{1mm}} \widehat{\bigoplus\limits_{\lambda\in \widehat{G}}} V_\lambda\otimes \left[V_{\lambda}^*\otimes E\right]^H,
		\end{equation}
		where $V_\lambda$ stands for the finite-dimensional irreducible $G$-representation of highest weight $\lambda\in\mathfrak{t}^*$ and $V_\lambda^*$ for its dual.
		Let
		\begin{equation*}
		\mathcal{D}^t_{G/H}(E)_\lambda: \left[L^2(G)\otimes(S\otimes E)\right]^H_{ \lambda}\rightarrow \left[L^2(G)\otimes(S\otimes E)\right]^H_{ \lambda}
		\end{equation*} be the restriction of the operator $\mathcal{D}^t_{G/H}(E)$ on the $\lambda$-isotypic component, $\lambda\in\widehat{G}$, and $D^t_{\mathfrak{g},\mathfrak{h}}(V_\lambda^*)$ the algebraic Dirac operator attached to $V_\lambda^*$ as in Section \ref{main}.
		Then the diagram

		\begin{equation*}
	\begindc{\commdiag}[900]
	\obj(0,1){$\left[L^2(G)\otimes(S\otimes E)\right]^H_{ \lambda} $}
	\obj(3,1){$\left[L^2(G)\otimes(S\otimes E)\right]^H_{ \lambda}$}
	\obj(0,0){$V_\lambda\otimes \left[V_{\lambda}^*\otimes (S\otimes E)\right]^H $}
	\obj(3,0){$V_\lambda\otimes \left[V_{\lambda}^*\otimes (S\otimes E)\right]^H$}
	\mor{$\left[L^2(G)\otimes(S\otimes E)\right]^H_{ \lambda} $}{$\left[L^2(G)\otimes(S\otimes E)\right]^H_{ \lambda}$}{$\mathcal{D}^t_{G/H}(E)_\lambda$}
	\mor{$\left[L^2(G)\otimes(S\otimes E)\right]^H_{ \lambda} $}{$V_\lambda\otimes \left[V_{\lambda}^*\otimes (S\otimes E)\right]^H $}{$\Phi$}
	\mor{$\left[L^2(G)\otimes(S\otimes E)\right]^H_{ \lambda}$}{$V_\lambda\otimes \left[V_{\lambda}^*\otimes (S\otimes E)\right]^H$}{$\Phi$}
	\mor{$V_\lambda\otimes \left[V_{\lambda}^*\otimes (S\otimes E)\right]^H $}{$V_\lambda\otimes \left[V_{\lambda}^*\otimes (S\otimes E)\right]^H$}{$1\otimes D^t_{\mathfrak{g},\mathfrak{h}}(V_\lambda^*)\otimes 1$}
	\enddc
		\end{equation*}
		is commutative, i.e. 
		\begin{equation*}
		\Phi\circ\mathcal{D}^t_{G/H}(E)_\lambda=(1\otimes D^t_{\mathfrak{g},\mathfrak{h}}(V_\lambda^*)\otimes 1)\circ \Phi.
		\end{equation*}
		Therefore, $\Phi$
		establishes an equivalence
		\begin{equation}
		\ker\mathcal{D}^t_{G/H}(E)_\lambda\simeq V_\lambda\otimes \left[\ker D^t_{\mathfrak{g},\mathfrak{h}}(V_\lambda^*)\otimes E\right]^H
		\end{equation}
		of representations of $G$.
		Since $G$ acts trivially on the component $[V_\lambda^*\otimes (S\otimes E)]^H$ of $V_\lambda\otimes[V_\lambda^*\otimes (S\otimes E)]^H$, $V_\lambda\otimes[\ker D^t_{\mathfrak{g},\mathfrak{h}}(V_\lambda^*)\otimes E]^H$ turns out to be equivalent to a direct sum of $\dim[\ker D^t_{\mathfrak{g},\mathfrak{h}}(V_\lambda^*)\otimes E]^H$ copies of $V_\lambda$, i.e. a direct sum of $\dim\mathrm{Hom}_H\big(E^*,\ker D^t_{\mathfrak{g},\mathfrak{h}}(V_\lambda^*)\big)$ copies of $V_\lambda$ which, for the sake of simplicity, we denote 
		\begin{equation*}
		\dim\left[ \mathrm{Hom}_H\big(E^*,\ker D^t_{\mathfrak{g},\mathfrak{h}}(V_\lambda^*)\big)\right]V_\lambda.
		\end{equation*}
		On the other hand, the spin module $S$ is self-dual, i.e. it coincides with its dual (see \cite{Deligne}), and thus $D^t_{\mathfrak{g},\mathfrak{h}}(V_\lambda^*)$ can be seen as the dual operator $D^t_{\mathfrak{g},\mathfrak{h}}(V_\lambda)^*$ of $D^t_{\mathfrak{g},\mathfrak{h}}(V_\lambda)$ acting on $V_\lambda^*\otimes S^*$. 
		Then 
		\begin{align*}
		\mathrm{Hom}_H\big(E^*,\ker D^t_{\mathfrak{g},\mathfrak{h}}(V_\lambda^*)\big)&\simeq \mathrm{Hom}_H\big(E^*,\ker D^t_{\mathfrak{g},\mathfrak{h}}(V_\lambda)^*\big)\\
		&=\mathrm{Hom}_H\big(E^*,\big(\ker D^t_{\mathfrak{g},\mathfrak{h}}(V_\lambda)\big)^*\big)\\
		&=\mathrm{Hom}_H\big(\ker D^t_{\mathfrak{g},\mathfrak{h}}(V_\lambda),E\big)
		\end{align*}
		and thus
		\begin{equation*}\label{withlambda}
		\ker\mathcal{D}^t_{G/H}(E)_\lambda\simeq \dim\mathrm{Hom}_H\big(\ker D^t_{\mathfrak{g},\mathfrak{h}}(V_\lambda),E\big) V_\lambda
		\end{equation*}
		and 
		\begin{equation}\label{geomkernel}
		\ker \mathcal{D}^t_{G/H}(E)\simeq\bigoplus\limits_{\lambda\in\widehat{G}}\dim\mathrm{Hom}_H\big(\ker D^t_{\mathfrak{g},\mathfrak{h}}(V_\lambda),E\big) V_\lambda.
		\end{equation}
		For more details on the relationship between algebraic and geometric Dirac operators, see \cite{mehdizierauAdv} and \cite{mehdi-zierauColl}.
		
		\begin{theorem}\label{thm62} Let $G$ be a compact connected semisimple Lie group with complexified Lie algebra $\mathfrak{g}$, $H$ be a closed subgroup of $G$ such that its complexified Lie algebra $\mathfrak{h}$ satisfies condition \eqref{conditionh}. Suppose that there is a common Cartan subalgebra $\mathfrak{t}$ of $\mathfrak{g}$ and $\mathfrak{h}$, and $E$ is an irreducible finite-dimensional $\mathfrak{h}$-representation of highest weight $\mu\in\mathfrak{t}^*$ with respect to a fixed positive root system $\Delta^+_\mathfrak{h}$ such that $S\otimes E$ lifts to a representation of $H$. Then, for $t\in(0,2)$, the kernel $\ker\mathcal{D}^t_{G/H}(E)$ of the geometric $t$-noncubic Dirac operator
			\begin{equation*}
			\mathcal{D}^t_{G/H}(E):\left[L^2(G)\otimes (S\otimes E)\right]^H\rightarrow \left[L^2(G)\otimes (S\otimes E)\right]^H
			\end{equation*}
			coincides with the kernel $\ker\mathcal{D}_{G/H}(E)$ of the geometric cubic Dirac operator $\mathcal{D}_{G/H}(E)$. In particular, if there is $w\in W$ such that $w(\mu+\rho_\mathfrak{h})-\rho$ is analytically integral and dominant with respect to a positive system $\Delta^+\supseteq\Delta^+_\mathfrak{h}$ of $\Delta$, $\ker\mathcal{D}^t_{G/H}(E)$ is equivalent with the irreducible representation $V_{w(\mu+\rho_\mathfrak{h})-\rho}$ of highest weight $w(\mu+\rho_\mathfrak{h})-\rho$ of $G$. Otherwise, $\ker\mathcal{D}^t_{G/H}(E)$ is trivial.
		\end{theorem}
		
		\begin{proof}
			According to Theorem \ref{kerneltcubic}, for $t\in(0,2)$, $\ker D^t_{\mathfrak{g},\mathfrak{h}}(V_\lambda)=\ker D_{\mathfrak{g},\mathfrak{h}}(V_\lambda)$ and thus
			\begin{equation*}
			\dim\mathrm{Hom}_H\big(\ker D^t_{\mathfrak{g},\mathfrak{h}}(V_\lambda),E\big)=\dim\mathrm{Hom}_H\big(\ker D_{\mathfrak{g},\mathfrak{h}}(V_\lambda),E\big)
			\end{equation*}
			and, from \eqref{geomkernel},
			\begin{equation*}
			\ker\mathcal{D}^t_{G/H}(E)_\lambda=\ker\mathcal{D}_{G/H}(E)_\lambda.
			\end{equation*}
			Hence for every $t\in(0,2)$,
			\begin{equation*}
			\ker \mathcal{D}^t_{G/H}(E)=\ker\mathcal{D}_{G/H}(E).
			\end{equation*}
			Now the result follows from \cite{landweber}. 
		\end{proof}
		
		For the noncubic geometric Dirac opearator we have the following theorem.
		\begin{theorem}\label{thm63}
			Let $G$ be a compact connected semisimple Lie group with complexified Lie algebra $\mathfrak{g}$ and $T$ be a maximal torus of $G$ with complexified Lie algebra $\mathfrak{t}$. Fix a one-dimensional representation $E$ of $\mathfrak{t}$ determined by a $\Delta^+$-dominant element $\mu$ of ${t}^*$ with respect to a fixed positive root system $\Delta^+$ for $\Delta:=\Delta(\mathfrak{g},\mathfrak{t})$ and such that $S\otimes E$ lifts to a representation of $T$.
			For every element $\lambda$ of the set $\Lambda$ of all $\Delta^+$-dominant and analytically integral elements of $\mathfrak{t}^*$, set 
			\begin{equation*}
			K_\lambda:=\{\rho-\sum\limits_{\alpha\in A}\alpha\mid A\subseteq\big(\mathbb{R}\{\lambda\}\big)^\perp\cap\Delta^+\text{ and } \rho-\sum\limits_{\alpha\in A}\alpha \text{ is } \Delta^+\text{-dominant}\}
			\end{equation*}
			and 
			\begin{equation*}
			\mathcal{A}_\mu:=\{(\lambda,\kappa)\in\Lambda\times K_\lambda\mid \lambda+\kappa=\mu\}.
			\end{equation*}
			Then, the kernel of the noncubic Dirac operator
			\begin{equation*} \widehat{\mathcal{D}}_{G/T}(E):\left[L^2(G)\otimes (S\otimes E)\right]^T\rightarrow\left[L^2(G)\otimes (S\otimes E)\right]^T
			\end{equation*}
			is 
			\begin{equation*}
			\ker\widehat{\mathcal{D}}_{G/T}(E)=\sum\limits_{(\lambda,\kappa)\in \mathcal{A}_\mu}\dim(V_\lambda\otimes S)_\mu V_{\lambda},
			\end{equation*}
			where $V_\lambda$ is the finite-dimensional irreducible $G$-representation of highest weight $\lambda$.
		\end{theorem}
		
		\begin{proof}
			According to Theorem \ref{thmbasic},
			\begin{align*}
			\ker \widehat{D}_{\mathfrak{g},\mathfrak{t}}(V_\lambda)&=\bigoplus\limits_{\nu\in A(\lambda)}(V_\lambda\otimes S)^{[\nu]}\\
			&=\bigoplus\limits_{\kappa\in K_\lambda}(V_\lambda\otimes S)^{[\lambda+\kappa]}\\
			&=\hspace{-4mm}\bigoplus\limits_{\substack{\kappa\in K_\lambda\\w\in W/W_{\lambda+\kappa}}}\hspace{-4mm}(V_\lambda\otimes S)_{w(\lambda+\kappa)}.
			\end{align*}
			Therefore, 
			\begin{align*}
			\ker\widehat{\mathcal{D}}_{G/T}(E)_\lambda&=\dim\mathrm{Hom}_T\big(\ker \widehat{D}_{\mathfrak{g},\mathfrak{t}}(V_\lambda),E\big) V_\lambda\\
			&=\dim\mathrm{Hom}_T\big(\hspace{-4mm}\bigoplus\limits_{\substack{\kappa\in K_\lambda\\w\in W/W_{\lambda+\kappa}}}\hspace{-4mm}(V_\lambda\otimes S)_{w(\lambda+\kappa)},E\big) V_\lambda\\
			&=\sum\limits_{\substack{\kappa\in K_\lambda\\w\in W/W_{\lambda+\kappa}}}\hspace{-4mm}\dim\mathrm{Hom}_T\big((V_\lambda\otimes S)_{w(\lambda+\kappa)},E\big) V_\lambda.
			\end{align*}
			Since $\mu$ has been chosen to be $\Delta^+$-dominant, 
			\begin{align*}
			\ker\widehat{\mathcal{D}}_{G/T}(E)_\lambda&=\sum\limits_{\kappa\in K_\lambda}\dim\mathrm{Hom}_T\big((V_\lambda\otimes S)_{\lambda+\kappa},E\big) V_\lambda\\
			&=\sum\limits_{\kappa\in K_\lambda}\dim(V_\lambda\otimes S)_{\lambda+\kappa}\cdot\delta_{\mu,\lambda+\kappa}\cdot V_\lambda\\
			&=\sum\limits_{\kappa\in K_\lambda}\dim(V_\lambda\otimes S)_{\mu}\cdot\delta_{\mu-\kappa,\lambda}\cdot V_{\mu-\kappa}
			\end{align*}
			Hence
			\begin{align*}
			\hspace{26mm}
			\ker\widehat{\mathcal{D}}_{G/T}(E)&=\sum\limits_{\substack{\lambda\in\Lambda\\ \kappa\in K_\lambda}}\dim(V_\lambda\otimes S)_{\mu}\cdot\delta_{\mu-\kappa,\lambda}\cdot V_{\mu-\kappa}&\\
			&=\sum\limits_{(\lambda,\kappa)\in\mathcal{A}_\mu} \dim(V_\lambda\otimes S)_{\mu}V_{\lambda}.&	\hspace{21mm}\qedhere
			\end{align*}
		\end{proof}

	\end{document}